\documentclass[12pt]{amsart}
\usepackage{amscd,graphicx}
\usepackage{pb-diagram}  

\newtheorem{thm}{Theorem}[section]
\newtheorem{prp}[thm]{Proposition}

\newtheorem{lma}[thm]{Lemma}
\theoremstyle{definition}

\newtheorem{exa}{Example}
\newtheorem{rem}{Remark}

\def \ph{\varphi}

\def\diff{\operatorname{d}}

\def \ra{\rightarrow}

\def \ie{\hbox{\it i.e.}}

\def \tns{\otimes}
\def \mtns{\tns\cdots\tns}

\def \mplus{+\cdots+}
\def \mcom{,\cdots,}
\def \k{\mbox{$\mathbb K$}}

\def \C{\mbox{$\mathbb C$}}

\def \Z{\mbox{$\mathbb Z$}}
\def \N{\mbox{$\mathbb N$}}

\def \MDer{\operatorname{MDer}}
\def \Der{\operatorname{Der}}

\def\zt{\mbox{$\Z_2$}}

\def\ad{\operatorname{ad}}

\def\im{\operatorname{Im}}
\def\A{\mbox{$\mathcal A$}}

\def\D{D}

\def\der{\operatorname{Der}}

\def\pinf{\mbox{$P_\infty$}}
\def\and{\mbox{ \rm and }}

\def\s#1{(-1)^{#1}}

\def\ho{\text{ho}}

\def\Aut{\operatorname{Aut}}

\def\px{\partial_x}
\def\py{\partial_y}
\def\pt{\partial_\theta}
\def\gcd{\operatorname{gcd}}

\def\Div{\operatorname{Div}}
\def\vn{\vec{\nabla}}
\def\del{\partial}
\def\cdits{\!\cdot\cdot\,}
\author{Michael Penkava}
\address{University of Wisconsin\\
Eau Claire, WI 54702-4004\\
USA}
\email{penkavmr@uwec.edu}
\author{Anne Pichereau}
\address{
Universit\'e de Lyon, Institut Camille Jordan (UMR 5208) \\
Universit\'e de Saint-Etienne, Facult\'{e} des Sciences\\
 23, Rue Docteur Paul Michelon, 42023 Saint-Etienne
Cedex 2,
France} \email{anne.pichereau@univ-st-etienne.fr} 
\thanks{The research of the authors was partially supported by the
french ANR projet
ANR-09-RPDOC-009-01, and by grants
from the University of Wisconsin-Eau Claire}
\title[\zt-graded Poisson Algebras]{\zt-graded Poisson Algebras, their deformations and cohomology in low dimensions}

\begin{document}
\setlength{\multlinegap}{0pt}

\begin{abstract}
We study \zt-graded Poisson structures defined on \zt-graded commutative polynomial algebras. In small dimensional cases, we obtain the associated Poisson \zt-graded cohomology and in some cases, deformations of these Poisson brackets and \pinf-algebra structures. We highlight differences and analogies between this \zt-graded context and the non graded context, by studying for example the links between Poisson cohomology and singularities.
\end{abstract}

\date\today
\maketitle
\section{Introduction}

In \cite{Kont}, M. Kontsevich has shown that every Poisson manifold $(M,툎pi)$ admits a deformation quantization, i.e., there exists a formal deformation of the underlying associative product of $C^\infty(M)$, whose first term is the initial associative product and the second term is the Poisson bracket $\pi$. He also proved that there is a one-to-one correspondence between the formal deformations of the associative product whose second term is $\pi$ and the formal deformations of the Poisson structure $\pi$. The results obtained and the methods used in his paper have motivated a lot of research areas, especially the study of deformations and consequently the study of cohomology, as cohomology is strongly related to deformations (of associative products, as well as of Poisson brackets). The second cohomology space gives indeed important informations about classification of deformations, while the third cohomology space gives the obstruction to the construction, at each step, of the deformation.

It is in general difficult to determine  Poisson cohomology. It has been defined by A. Lichnerowicz in \cite{Lichne}, but see also \cite{Hueb} for an algebraic approach. In an algebraic context, that is when the algebra of smooth functions $C^\infty(M)$ is replaced by the polynomial algebra $\C[x_1, \dots,x_d]$, and in low dimensions ($d\leq 4$), it has been the subject of several publications: for example, when $d=2$ by P. Monnier in \cite{Monnier} (except that this work is in a germified context) and by C. Roger and P. Vanhaecke in \cite{RogerVanhaecke}, when $d=3$ by the second author in \cite{Pichereau1}, and when $d=4$ by S. Pelap in \cite{Pelap}, ... These different works have shown that even in an algebraic and low dimensional context, having explicit bases of the Poisson cohomology spaces is very interesting, as this cohomology gives information about the underlying variety but also about the singularities of the Poisson structure and, as said before, about its formal deformations (for example, results in \cite{Pichereau1} have been used in \cite{Pichereau2}).

In this paper, we study analogous questions but in a $\zt$-graded context, where the underlying algebra becomes $\C[x_1, \dots, x_m, \theta_1, \dots, \theta_n]$, with even variables $x_i$ and odd variables $\theta_j$. The algebra is then the algebra of functions on a supermanifold of dimension $m|n$ (see for example \cite{Catt-Scha}). When $m=0$, we get the classical case (the one considered above). This is motivated by physics, where bosonic and fermionic dynamical variables play the role of even and odd variables (see for example \cite{Rogers}).

The introduction of odd variables generates the introduction of additional signs. In section \ref{section2}, we give details about this $\zt$-graded context, definitions of the $\zt$-graded Poisson structures and their cohomology. But in sections \ref{section01}, \ref{section11} and \ref{section21}, we study these notions in low dimensional  cases ($m|n=0|1, 1|1, 2|1$) and analogously to the classical context, we obtain interesting results. In particular, we notice that, even if the differences between the definitions in classical and $\zt$-graded contexts lie essentially in signs, the conditions for the existence of (nontrivial) Poisson structures, as well as the explicit bases of Poisson cohomology spaces are sometimes very different from one of these contexts to the other one.

Considering the Poisson structures and their cohomology, here is a list of some of the differences between the classical and the $\zt$-graded cases that the reader could observe in this paper (in the rest of this introduction, $x, y, z$ denote even variables, $\theta$ denotes an odd variable, and $H^k$ or $H^k(\pi)$ denote the $k$-th Poisson cohomology space associated to the ($\zt$-graded or not) Poisson structure $\pi$):\\

%
%
{\bf 1.} On the classical polynomial algebra $\C[x]$, there is no nontrivial Poisson structure, while on $\C[\theta]$, there is a one dimensional vector space of odd $\zt$-graded biderivations and every odd $\zt$-graded biderivation is a ($\zt$-graded) Poisson structure.\\
%

{\bf 2.} On the classical polynomial algebra $\C[x,y]$, every biderivation $\{\cdot, \cdot\}^\psi :=\psi\, \partial_x\wedge \partial_y$ (with $\psi\in \C[x,y]$) is a Poisson structure, while on $\C[x,\theta]$, we will see in section \ref{section11} that there is a nontrivial condition for an odd $\zt$-graded biderivation to be a ($\zt$-graded) Poisson structure. \\
%

{\bf 3.} The $\zt$-grading implies a decomposition into odd and even parts for the Poisson cohomology that doesn't exist in the classical case.\\
%

{\bf 4.} On a classical polynomial algebra $\C[x_1,\dots, x_m]$, as there is no nontrivial skew-symmetric $k$-derivation (which are the cochains of the Poisson cohomology) if $k>m$, one has $H^k=0$ when $k>m$, while on a $\zt$-graded polynomial algebra, we will see in the examples in the sections below that there are non zero $k$-th Poisson cohomology spaces $H^k$ for all $k\in툎N$.\\
%
%

%
Despite these important differences of behaviour of Poisson structures and cohomology, between  the classical and $\zt$-graded cases, some (maybe) surprising and interesting similarities are observed, especially in the explicit bases of Poisson cohomology spaces. \\

%
%
{\bf 1.} To every $\ph\in \C[x,y,z]$, one can associate a Poisson structure $\pi_\ph:= \ph_x\, \partial_y\wedge\partial_z + \ph_y\, \partial_z\wedge\partial_x+ \ph_z\, \partial_x\wedge\partial_y$, defined on $\C[x,y,z]$, with the partial derivatives $\ph_x,\ph_y,\ph_z$ of $\ph$ ;\\
On the $\zt$-graded polynomial algebra $\C[x,y,\theta]$, to every $b\in \C[x,y]$, one associates a $\zt$-graded Poisson structure $\psi_b=b_y\, \theta\px\wedge\pt - b_x\, \theta\py\wedge\pt$ (see section \ref{section21}), similarly defined with the partial derivatives of the polynomial $b$.\\

\textit{(In the sequel, for results about the Poisson cohomology of the Poisson algebra $(\C[x,y], \{\cdot,\cdot\}^\psi)$, we refer to \cite{Monnier} and \cite{RogerVanhaecke} ; for results about the Poisson cohomology of $(\C[x,y,z], \pi_\ph)$, we refer to \cite{Pichereau1}, while for results about the ($\zt$-graded) Poisson cohomology of $(\C[x,y,\theta], \psi_b)$, see section \ref{section21}.)}\\

{\bf 2.} In the classical, as well as in the $\zt$-graded case, the singularity of the Poisson structure is closely related to the associated Poisson cohomology spaces. For example,\\
$\circ$ if $\psi\in \C[x,y]$ is homogeneous with an isolated singularity (that is, $\psi$ is square-free), then the Milnor algebra $\frac{\displaystyle\C[x,y]}{\displaystyle(\psi_x,\psi_y)}$ of the singularity of $\psi$ appears in the Poisson cohomology spaces of $(\C[x,y],\{\cdot, \cdot\}^\psi)$ ;\\
$\circ$  if $\ph\in \C[x,y,z]$ is homogeneous with an isolated singularity, then the Milnor algebra $\frac{\displaystyle \C[x,y,z]}{\displaystyle(\ph_x,\ph_y,\ph_z)}$ of the singularity of $\ph$ appears in the Poisson cohomology spaces of $(\C[x,y,z],\pi_\ph)$ ;\\
$\circ$ while if $b\in \C[x,y]$ is homogeneous and square-free, then the Milnor algebra $\frac{\displaystyle\C[x,y]}{\displaystyle(b_x,b_y)}$ of the singularity of $b$ appears in the Poisson cohomology spaces of $(\C[x,y,\theta],\psi_b)$.\\
%

{\bf 3.} By definition of the Poisson cohomology, the Poisson structure is always a $2$-cocycle for its associated cohomology. In the three following examples, the condition for the ($\zt$-graded or not) Poisson structure to be a $2$-coboundary is very similar:\\
$\circ$ In $\C[x,y]$, the Poisson structure $\{\cdot,\cdot\}^\psi$ (with $\psi \in \C[x,y]$ homogeneous and square-free) is a coboundary iff $\deg(\psi)\not=2$ ;\\
$\circ$ In $\C[x,y,z]$, the Poisson structure $\pi_\ph$ (with $\ph\in \C[x,y,z]$ homogeneous with an isolated singularity) is a coboundary iff $\deg(\ph)\not=3$ ;\\
$\circ$ In $\C[x,y,\theta]$, the Poisson structure $\psi_b$ (with $b$ homogeneous and square-free) is a coboundary  iff $\deg(b)\not=2$.\\

{\bf 4.} Observing the explicit bases obtained for Poisson cohomology spaces of $(\C[x,y,z], \pi_\ph)$ and $(\C[x,y,\theta], \psi_b)$, one can see that some pieces are very similar, \\
$\circ$ for example, the part ``$\mathrm{Cas}\, \vec E$"  in $H^1(\pi_\ph)$ or $H^1_e(\psi_b)$ (where $\vec E$ is the Euler derivation $x\partial_x+y\partial_y+z\partial_z$  of $\C[x,y,z]$ or the corresponding derivation $x\px+ y\py$ of $\C[x,y,\theta]$), appearing iff $\deg(\ph)=3$ or $\deg(b)=2$.\\
$\circ$ Also, in the Poisson cohomology associated to  $(\C[x,y,z],\pi_\ph)$, as well as $(\C[x,y,\theta],\psi_b)$, the third Poisson cohomology space $H^3(\pi_\ph)$ or $H^3_e(\psi_b)$ is isomorphic to the corresponding Milnor algebra (of $\ph$ or $b$), tensored with the Casimirs.\\
$\circ$ Finally, denoting by $\{u_i\}_i$ a (homogeneous) basis of the Milnor algebra associated to $\ph$, respectively to $b$, then the elements of the form $u_j\,\pi_\ph$ and $\pi_{u_l}$, respectively $u_j\,\psi_b$ and $\psi_{u_l}$ appear in a basis of $H^2(\pi_\ph)$, respectively $H^2_o(\psi_b)$, under  similar conditions on $j$ and $l$.\\

Finally, the notion of $\zt$-graded Poisson structures extends to the notion  of $P_\infty$ structure. This is the subject of section \ref{sectionPinf}, where we introduced these structures, their cohomology and their moduli spaces, and give results in the low dimensional cases $0|1$ and $1|1$. The connection between $P_\infty$ structures and Poisson cohomology is parallel to the connections between $A_\infty$ structures and Hochshild cohomology of associative algebras, or $L_\infty$ structures and the Eilenberg-Chevalley cohomology of Lie algebras, so arises in a natural context in the study of Poisson algebras.
 
{\bf Acknowledgements.} We would like to warmly thank the referees for helpful remarks and comments.


%
%
\section{Preliminaries}
\label{section2}
Let $V=\langle x_1\mcom x_m,\theta_1\mcom \theta_n\rangle$ be
an $m|n$-dimensional \zt-graded vector space, over a (non graded) field $\k$, which
for simplicity, we will assume to be of characteristic zero,
and sometimes will assume to be algebraically closed.
For each $1\leq i \leq m$, the parity of $x_i$ is $0 \mod 2$ and denoted by $|x_i|$, while for each $1\leq j \leq n$, the parity of the variable $\theta_j$ is $1 \mod 2$ and denoted by $|\theta_j|$.
The
polynomial algebra $\A=\k[x_1\mcom x_m,\theta_1\mcom\theta_n]$ is
just the symmetric algebra $S(V)$, which is
\zt-graded commutative, that is $S(V)=T(V)/I$, where $I$ is the two-sided ideal generated by the elements of the form $a\otimes b -(-1)^{|a||b|}b\otimes a$, with $a$ and $b$ homogeneous elements of $V$.

The parity of a monomial $ax_{i_1}\cdots x_{i_k}\theta_{j_1}\cdots\theta_{j_l}$
is  $l\mod 2$. A polynomial $f$  in $\A$ is said to be homogeneous if it is a sum of
monomials of the same parity, which we call the parity of $f$ and denote by
$|f|$. If $f$ and $g$ are homogeneous elements in $\A$, then $fg$ is homogeneous
of parity $|fg|=|f|+|g|$, and $gf=\s{fg}fg$, where, $\s{fg}$ is a shorthand
notation for $\s{|f||g|}$.

Let $\partial_{x_1}\mcom \partial_{x_m},\partial_{\theta_1}\mcom\partial_{\theta_n}$ be the dual basis of
$V^*$. An element of $V^*$ extends uniquely to a \zt-graded derivation of $\A$. Moreover, we
can identify the space $\der(\A)$ of \zt-graded derivations of $\A$ with $\A\tns V^*$.

The space $\A\tns\bigwedge(V^*)$ is naturally identified with
the space $\MDer(\A)$ of multiderivations of $\A$, that is
the skew-symmetric multilinear maps $\bigwedge\A\ra\A$ which are derivations in each argument\footnote{Notice that in $\bigwedge(V^*)$, as well as in $\bigwedge(\A)$, the wedge product is $\zt$-graded skew-symmetric, for example $\bigwedge(\A)=T(\A)/J$, where $J$ is the two-sided ideal generated by the elements of the form $f\otimes g+(-1)^{|f||g|}g\otimes f$, with $f$ and $g$ homogeneous elements of $\A$.}.
In fact, if
$\alpha=f\partial_{u_1}\wedge\cdots\wedge\partial_{u_k}\in \A\tns\bigwedge^k(V^*)$,
where $f\in\A$ and $u_i\in V^*$, then if $g_1\wedge\cdots \wedge g_k\in\bigwedge^k(\A)$ (with $g_1, \dots, g_k$ homogeneous elements of $\A$), we define
\begin{equation*}
\alpha(g_1\wedge\cdots\wedge g_k)=f\sum_{\sigma\in S_k}\s{\sigma}\epsilon(\sigma)\s{g\cdot\sigma}
\partial_{u_1}(g_{\sigma(1)})\wedge\cdots\wedge\partial_{u_k}(g_{\sigma(k)}).
\end{equation*}
Here $\s{\sigma}$ is the signature of the permutation $\sigma$,
$\epsilon(\sigma, g)$ is a sign depending on $\sigma$ and $g:=(g_1\mcom g_k)$, which is determined by the following equality in $\bigwedge\A$
\begin{equation*}
\epsilon(\sigma,g)g_{\sigma(1)}\wedge\cdots\wedge g_{\sigma(k)}=g_1\wedge\cdots\wedge g_k,
\end{equation*}
and $\s{g\cdot\sigma}$ is the sign given by
\begin{equation}
g\cdot\sigma=u_2g_{\sigma(1)}+u_3(g_{\sigma(1)}+g_{\sigma(2)})
\mplus
u_k(g_{\sigma(1)}\mplus g_{\sigma(k-1)}),
\end{equation}
which results from the Koszul's sign appearing when applying the tensor
$\partial_{u_1}\mtns\partial_{u_k}$ to $g_{\sigma(1)}\mtns g_{\sigma(k)}$.

The bidegree of $\alpha$ is $(\deg(\alpha),|\alpha|)$ where
$\deg(\alpha)=k-1$ is the exterior degree of $\alpha$, that is, its
degree in terms of the \Z-grading of $\bigwedge(V^*)$, and
$|\alpha|=|f|+|u_1|\mplus |u_k|$ is the parity of $\alpha$ as a
\zt-graded map.
%

Because the space $\MDer(\A)$ is  identified with  $\A\tns\bigwedge(V^*)$ and  $\der(\A)$ with $\A\tns V^*$, every $\zt$-graded $k$-derivation of $\A$ can be seen as an element of
$\bigwedge^k(\Der(\A))$.
 Denote $C^k=\bigwedge^k(\Der(\A))$,
so that $\MDer(\A)=\prod_{k=0}^\infty C^k$. In order to simplify the notations, an element $\delta_1\wedge\cdots\wedge\delta_k\in C^k$, where $\delta_1, \dots, \delta_k\in \Der(\A)$ will often be denoted by $\delta_1\cdots\delta_k$ (where the wedge product has been omited).

If  $\delta=\delta_1\cdots\delta_k\in C^k$ and
$\mu=\mu_1\cdots\mu_l\in C^l$, then
\begin{align*}
[\delta,\mu]=
&\sum_{i=1}^k\sum_{j=1}^l
\s{i+j+\star}
[\delta_i,\mu_j]
\delta_1\cdots\hat{\delta_i}\cdots\delta_k\mu_1\cdots\hat{\mu_j}\cdots\mu_l,
\end{align*}
where
\begin{equation*}
\star=\delta_i(\delta_1\mplus\delta_{i-1})+\mu_j(\delta_1\mplus\hat{\delta_i}\mplus
\delta_k+\mu_1\mplus\mu_{j-1}),
\end{equation*}
and
\begin{equation*}
[\delta_i,\mu_j]=\delta_i\circ\mu_j-\s{\delta_i\mu_j}\mu_j\circ\delta_i
\end{equation*}
is the usual bracket of $\delta_i$ and $\mu_j$ as derivations of $\A$,
which coincides with their bracket as coderivations of $\bigwedge(\A)$.
In particular,
\begin{equation*}
[f\del_u,g\del_v]=f\del_u(g)\del_v+\s{(u+f)(v+g)+1}g\del_v(f)\del_u.
\end{equation*}
As a consequence, we obtain that
\begin{align*}
[f\del_{u_1}\cdits\del_{u_m},g\del_{v_1}\cdits\del_{v_n}]=&
\sum_{i=1}^m\s{m-i+\star}
f\del_{u_i}(g)\del_{u_1}\cdits\widehat{\del_{u_i}}\cdits\del_{u_m}\del_{v_1}\cdits\del_{v_n}\\
&+\sum_{j=1}^n\s{j+\star\star}g\del_{v_j}(f)\del_{u_1}\cdits\del_{u_m}\del_{v_1}\cdits
\widehat{\del_{v_j}}\cdits\del_{v_n}
\end{align*}
where
\begin{align*}
\star&=u_i(u_1\mplus u_{i-1})+g(u_1\mplus\widehat{u_i}\mplus u_m)\\
\star\star&=(f+u_1\mplus u_m)(v_j+g)+v_j(v_1\mplus v_{j-1}).
\end{align*}

The bracket above is known as the Schouten bracket (see for example \cite{Azca}), and it
equips the \zt-graded multiderivations of $\A$
with the structure of a $\Z\times\zt$-graded
Lie algebra. This type of grading is unsatisfactory from the point of view
of deformation theory, and there is a standard method to convert it to
a \zt-grading.  To do this, we introduce the \emph{modified Schouten bracket}, given
by
\begin{equation*}
\{\alpha,\beta\}=\s{deg(\alpha)|\beta|}[\alpha,\beta]
\end{equation*}
if $\alpha$ and $\beta$ are homogeneous in bidegree. This bracket equips $C=\MDer(\A)$ with the
structure of a \zt-graded Lie algebra, rather than bigraded algebra, where the \zt-graded
degree of $\alpha$ becomes $\deg(\alpha)+|\alpha|\pmod 2$. About this modified bracket, see \cite{Stasheff}, \cite{Penkava1} and \cite{Penkava2}.
 Denote by $C_e$ and $C_o$ the
subspaces of $C$ consisting of even and odd multiderivations, respectively, and similarly
for $C^k_e$ and $C^k_o$.

A \emph{($\zt$-graded) Poisson algebra structure} $\psi$ on $\A$ is an odd biderivation $\psi\in C^2_o$ satisfying $[\psi,\psi]=0$.
In fact, if we define a bracket on $\A$ by $[a,b]=\psi(a,b)$, then $\psi$ is a
Poisson structure on $\A$ precisely when the following conditions occur:
\begin{align*}
[a,[b,c]]&=[[a,b],c]+\s{ab}[b,[a,c]],&\quad\text{The graded Jacobi Identity}\\
[a,bc]&=[a,b]c+\s{ab}b[a,c],&\quad\text{The graded Leibniz Identity}
\end{align*}
The formulas above make sense when $\A$ is any graded algebra. In this paper, we
will be studying cases where $\A$ is a \zt-graded commutative algebra, but it is interesting
to note that the definition of a Poisson algebra given above does not depend on
the \zt-graded commutativity of the associative algebra structure on \A.

Suppose that $\psi$ is a Poisson structure on $\A$ (which implies $\deg( \psi)=1$, $|\psi|=0$). We define the operator $D:C\to C$ by $D(\alpha)=\lbrace \psi, \alpha\rbrace$ for all $\alpha\in C$. The graded Jacobi identity for $\lbrace \cdot, \cdot \rbrace$ and $\lbrace \psi, \psi\rbrace=0$ imply $D\circ D=0$, so that we can define the Poisson cohomology $H(\psi)$ associated to $\psi$ by $H(\psi) = \ker(D)/\im(D) = \bigoplus\limits_{n\geq 0} H^n(\psi)$, where
$$
H^n(\psi) = \ker(D:C^n\ra C^{n+1})/\im(D:C^{n-1}\ra C^n).
$$
The cohomology inherits a natural grading, so we have a decomposition $H(\psi)=H_o(\psi)\oplus H_e(\psi)$, because for homogeneous (in the bigraded sense) $\alpha\in C$, one has $\deg(D(\alpha))=\deg(\alpha)+1$ and $|D(\alpha)| = |\alpha|$.

\section{Poisson Structures on a $0|1$-dimensional polynomial algebra}
\label{section01}
Poisson structures on a $1|0$-dimensional algebra $\k[x]$ are all trivial, so not
interesting. The same is not true for Poisson structures on a $0|1$-dimensional algebra,
which arises as the algebra of a supermanifold over a singleton point.

Let $\A=\k[\theta]=\k\oplus\k\theta$ be the $0|1$-dimensional polynomial algebra.
Let
\begin{align*}
\ph^n&=\theta\del_\theta^n\\
\psi^n&=\del_\theta^n
\end{align*}
Then $C^n_o=\langle\psi^n\rangle$ and $C^n_e=\langle\ph^n\rangle$ are the subspaces of $C^n$
of odd and even elements respectively.
Let now $\psi\in C^2_o$, that is $\psi \in \k\,\del_\theta\wedge\del_\theta$. As $\lbrack \psi, \psi\rbrack=0$, every odd biderivation of $\A$ is a Poisson structure on $\A$.
It is straightforward to show that
$$
[\psi,\psi^n]=0,\qquad
[\psi,\ph^n]=2\psi^{n+1}.
$$
So that
$$
\{\psi,\psi^n\}=0,\qquad
\{\psi,\ph^n\}=(-1)^{(n+1)}2\psi^{n+1}.
$$
It is also easy to see that
\begin{align*}
H^n(\psi)=
\begin{cases}
\k,&n=0\\
0,&n\ge 1.
\end{cases}
\end{align*}
From this, we easily deduce that up to isomorphism, there is a unique nontrivial Poisson algebra structure on a $0|1$-dimensional space, given by 
$\psi=\del_\theta\wedge\del_\theta$.

\section{Poisson Structures on a $1|1$-dimensional polynomial algebra}
\label{section11}
We consider Poisson structures on $\A=\k[x,\theta]$, the $1|1$-dimensional polynomial algebra.
Let
\begin{align*}
\psi^k=\psi^k(f_k,g_k)&=f_k(x)\theta\px\pt^{k-1}+g_k(x)\pt^k,\\
\ph^k=\ph^k(a_k,b_k)&=a_k(x)\px\pt^{k-1}+b_k(x)\theta\pt^{k}.
\end{align*}
Then
\begin{align*}
C^k_o&=\lbrace\psi^k(f_k,g_k)\mid f_k,g_k\in\k[x]\rbrace,\\
C^k_e&=\lbrace\ph^k(a_k,b_k)\mid a_k,b_k\in\k[x]\rbrace.
\end{align*}

One can check that for $\psi = f(x)\theta\px\wedge\pt+g(x)\pt\wedge\pt\in C^2_o$ (where $f,g\in \k[x]$),
$$
\{\psi, \psi\} =-4fg\px\pt^2-2fg'\theta\pt^3,
$$
so that $\psi$ is a Poisson structure on $\A$ iff $f=0$ or $g=0$.
We shall call a Poisson structure of the form $g(x)\pt\wedge\pt$ a Poisson structure \emph{of the first kind}, and one of the form $f(x)\theta\px\wedge\pt$ a Poisson structure \emph{of the second kind}.\\

\subsection{Poisson structures of the first kind}
\label{section11first}
Let
$
\psi=g(x)\pt\wedge \pt
$
be a nontrivial Poisson structure on $\A=\k[x,\theta]$ of the first kind.
Let us compute the action of the coboundary operator $D$,
given by $D(\phi)=\{\psi,\phi\}$ on cochains. We can
use the $\ph_l$ as a basis of the even cochains, but need to
introduce a notation for an odd cochain $\alpha^l$,
given by
$
\alpha^l(s_l,t_l)=s_l(x)\theta\px\pt^{l-1}+t_l(x)\pt^l
$.
We compute
\begin{align*}
[\psi,\alpha^l]=-2s_lg\px\pt^{l}-s_lg'\theta\pt^{l+1}.
\end{align*}
It follows that an odd cochain $\alpha=c^l\alpha_l$ is a $D$-cocycle
 precisely when $s_l=0$ for all $l$. Moreover, we have
\begin{equation*}
[\psi,\ph^l]=(2gb_l-a_lg')\pt^{l+1},
\end{equation*}
so if $\ph=\ph^l(a_l,b_l)$ is an even cochain,
then the condition for $\ph$ to be a cocycle is
\begin{equation*}
2gb_{n-1}-a_{n-1}g'=0\qquad n=0\dots
\end{equation*}
Note that $\ph^0=b_0(x)\theta$; \ie, $a_0=0$.
As $g\not=0$, the equation for $n=1$ gives $b_0=0$.
Also, $\alpha^0=t_0(x)$, so that automatically, $\alpha^0$ is always
a cocycle.
Accordingly, we have $H^0(\psi)=\k[x]$.

In order to compute $H^n(\psi)$ for $n> 0$, let $h(x)=\gcd(g(x),g'(x))$. Then $h$ measures the
\emph{singularity} of the Poisson structure $\psi$. We first compute
$H^n_o$, the odd part of the cohomology. Now the odd cocycles $Z^n_o$ and the
odd coboundaries $B^n_o$ are given by
\begin{align*}
Z^n_o&=\k[x]\pt^n\\
B^n_o&=
\begin{cases}
0  &n=0,\\
 2g(x)\k[x]\pt&n=1,\\
(2g(x)\k[x]+g'(x)\k[x])\pt^n&n\ge 2.
\end{cases}
\end{align*}
It follows that
\begin{equation*}
B^n_o=
h(x)\k[x]\pt^n,\quad \hbox{if } n\ge 2,
\end{equation*}
so we have
\begin{equation*}
H^n_o(\psi)=
\begin{cases}
\k[x]& n=0,\\
\vspace*{-0.3cm}
{}\\
\displaystyle{\frac{\k[x]}{(g(x))}}\, \pt &n=1,\\
\vspace*{-0.3cm}
{}\\
\displaystyle{\frac{\k[x]}{(h(x))}}\, \pt^{n} &n\ge 2,
\end{cases}
\end{equation*}
where for $a(x)\in \k[x]$, $(a(x))$ has to be understood as $\{a(x)b(x)\mid b(x)\in \k[x]\}$. Notice that $H^n(\psi)$ is a vector space,  in general not a ring, but in some good cases, it will be given by the quotient with the set of all multiples of an element $a(x)$, denoted by $(a(x))$.

Recall that for deformation theory, we normally do not include the odd 0-cochains, as a consequence,
it is natural to interpret $H^{1}_o=Z^{1}_o$. From this point of view, we
obtain $H^{1}_o=\k[x]\pt$.

To calculate the even part of the cohomology, suppose
$$\ph^n=a(x)\px\pt^{n-1}+b(x)\theta\pt^n$$ is an $n$-cocycle. The cocycle
condition for $\ph^n$ is $2b(x)g(x)=a(x)g'(x)$. Express $g(x)=p(x)h(x)$ and
$g'(x)=q(x)h(x)$, so $p(x)$ and $q(x)$ are relatively prime.  When $n>0$, the
cocycle condition reduces to
\begin{align*}
a(x)&=2p(x)m(x)\\
b(x)&=q(x)m(x),
\end{align*}
for an arbitrary $m(x)\in\k[x]$.  Now,
if $$\alpha_{n-1}=s(x)\theta\px\pt^{n-2}+t(x)\pt^{n-1},$$ then
\begin{equation*}
D(\alpha_{n-1})=-2s(x)g(x)\px\pt^{n-1}-s(x)g'(x)\theta\pt^n.
\end{equation*}
If we express $m(x)=u(x)h(x)+r(x)$ where $\deg(r(x))<\deg(h(x))$, then
\begin{align*}
a(x)&=2g(x)u(x)+2p(x)r(x)\\
b(x)&=g'(x)u(x)+q(x)r(x)
\end{align*}
gives a decomposition of $a(x)$ and $b(x)$
into terms coming from trivial and nontrivial cocycles.
When $n=1$, there are no $n$-coboundaries.
When $n=0$, we must have $a(x)=0$, so the cocycle condition
becomes $b(x)=0$. Thus there are no even 0-cocycles. Thus we have
\begin{align*}
H^n_e(\psi)=
\begin{cases}
0&n=0,\\
\k[x](2p(x)\px+q(x)\theta\pt)&n=1,\\
\vspace*{-0.3cm}
{}\\
\displaystyle{\frac{\k[x]}{(h(x))}}(2p(x)\,\px\pt^{n-1}+q(x)\theta\,\pt^n)&n\ge 2.
\end{cases}
\end{align*}

\subsection{Poisson structures of the second kind}
\label{section11second}
Let us suppose that
$
\psi=f(x)\theta\px\wedge\pt
$
is a nontrivial Poisson structure of the second kind.
If $\alpha^l(s_l,t_l)=s_l(x)\theta\px\pt^{l-1}+t_l(x)\pt^l
$, then
\begin{align*}
[\psi,\alpha^l]=-lft_l\px\pt^{l}-ft'_l\theta\pt^{l+1}.
\end{align*}
An odd cochain $\alpha=c^l\alpha_l$ is a $D$-cocycle
 precisely when $t_l=0$ for all $l>0$, and $t_0$ is a constant.
As a consequence, we have $H^0_o=\k$.
Moreover, we have
\begin{equation*}
[\psi,\ph^l]=(fa'_l-f'a_l+(1-l)fb_l)\theta\px\pt^{l},
\end{equation*}
so if $\ph=\ph^l(a_l,b_l)$ is an even cochain,
then the condition for $\ph$ to be a cocycle is
\begin{equation*}
fa'_{n-1}-f'a_{n-1}+(2-n)fb_{n-1}=0\qquad n=0\dots
\end{equation*}
As $f\not=0$ and $a_0=0$, applying the equation above with $n=1$, we obtain $b_0=0$ for any even cocycle $\ph^l$. Thus, $H^0_e=0$.

In order to determine $H^n(\psi)$ for $n\geq 1$, let $h(x)=\gcd(f(x),f'(x))$. Then as in
the case of Poisson structures of the first kind, we say that $h$ measures the
\emph{singularity} of the Poisson structure $\psi$.

We first compute
$H^n_o$, the odd part of the cohomology. If $n\ge1$, the odd cocycles $Z^n_o$ and the
odd coboundaries $B^n_o$ are given by
\begin{align*}
Z^n_o&=\k[x]\,\theta\px\pt^{n-1}\\
B^n_o&=
\begin{cases}
0&n=0,\\
f\k[x]\theta\px&n=1,\\
\{(fa'-f'a)\theta\px\pt\mid a\in \k[x]\}&n=2,\\
\{ (fa'-f'a+(2-n)fb)\theta\px\pt\mid a, b\in \k[x]\}&
\text{otherwise}.
\end{cases}
\end{align*}
It follows that
\begin{equation*}
B^n_o=
h(x)\k[x]\theta\px\pt^{n-1},\quad n\ge 3,
\end{equation*}
so we have
\begin{equation*}
H^n_o(\psi)=
\begin{cases}
\k & n=0,\\
\vspace*{-0.3cm}
{}\\
\displaystyle{\frac{\k[x]}{(f(x))}}\,\theta\px&n=1,\\
\vspace*{-0.3cm}
{}\\
\displaystyle{\frac{\k[x]}{\{fa'-af'\mid a\in \k[x]\}}}\, \theta\px\pt\; &n=2,\\
\vspace*{-0.3cm}
{}\\
\displaystyle{\frac{\k[x]}{(h(x))}}\,\theta\px\pt^{n-1}&n\ge 3.
\end{cases}
\end{equation*}
As in the case of Poisson structures of the first kind, we should omit
the coboundaries of 0-cochains, so for deformation purposes we have
$H^{1}_o(\psi)=\k[x]\theta\px$.

Notice that, the space $H^2_o(\psi)$ is very different for a Poisson structure of the second kind $\psi$, than for a Poisson structure of the first kind. In the case of a Poisson structure $\psi$ of the second kind, $B^{2(k-1)}_o(\psi)$ is indeed not given by the multiples of one polynomial anymore (unless $f=ax^m$ is a single term polynomial).

To calculate the even part of the cohomology, suppose
$$\ph^n=a(x)\px\pt^{n-1}+b(x)\theta\pt^n$$ is an $n$-cocycle. The cocycle
condition for $\ph^n$ is $f(x)(a'(x)+(1-n)b(x)=f'(x)a(x)$. Express $f(x)=p(x)h(x)$ and
$f'(x)=q(x)h(x)$, so $p(x)$ and $q(x)$ are relatively prime.  When $n\ne 1$, the
cocycle condition reduces to
\begin{align*}
a(x)&=m(x)p(x)\\
a'(x)+(1-n)b(x)&=m(x)q(x),
\end{align*}
for an arbitrary $m(x)\in\k[x]$.
Now,
if $$\alpha_{n-1}=s(x)\theta\px\pt^{n-2}+t(x)\pt^{n-1},$$ then
\begin{equation*}
D(\alpha_{n-1})=-(n-1)f(x)t(x)\px\pt^{n-1}-f(x)t'(x)\theta\pt^n.
\end{equation*}
If $n\ne 1$, then
we can express $m(x)=(n-1)u(x)h(x)+r(x)$ where $\deg(r(x))<\deg(h(x))$.
\begin{align*}
a(x)&=(n-1)f(x)u(x)+p(x)r(x)\\
a'(x)+(1-n)b(x)&=(n-1)f'(x)u(x)+q(x)r(x)
\end{align*}
gives a natural decomposition of $a(x)$ into a part
coming from a trivial and a nontrivial cocycle. Now we express
\begin{equation*}
b(x)=f(x)u'(x)+\tfrac1{1-n}\left(r(x)q(x)-r'(x)p(x)-r(x)p'(x)\right)
\end{equation*}
which gives the corresponding decomposition of $b(x)$ into trivial and nontrivial parts.
Unlike the case for Poisson structures of the first kind, we cannot express the cohomology
in terms of products of a single generator. However, there is an isomorphism
$H^n_e(\psi)$ with $\k[x]/(h(x))$ because the nontrivial parts of the decomposition above
are determined by $r(x)$.

Let us finally calculate $ H^{n}_e(\psi)$, in the particular case where $n=1$. First, suppose that $\varphi^{1}=a(x)\px+b(x)\theta\pt$ satifies $D(\ph^1)=0$ which is equivalent to the condition: $f(x)a'(x)-a(x)f'(x)=0$. We suppose $a\not=0$ so that this condition implies that $a$ and $f$ have same degree. Moreover, as $f(x)=p(x)h(x)$ and $f'(x)=q(x)h(x)$, we get
$$
p(x)a'(x)-a(x)q(x) = 0.
$$
Because $p$ and $q$ are coprime, there exists $m(x)\in \k[x]$ such that $a(x)=m(x)p(x)$. We decompose $m(x)= u(x)h(x)+r(x)$, with $u(x), r(x)\in \k[x]$ and $\deg(r(x))<\deg(h(x))$, and get $a(x) = u(x)f(x)+r(x)p(x)$. But $\deg(p(x)r(x))<\deg(f(x))=\deg(a(x))$ so that $u(x)\not=0$ and $\deg(u(x)f(x))=\deg(a(x))=\deg(f(x))$, which implies that $u(x)\in \k$ is a constant. Then,
\begin{eqnarray*}
0&=& f(x)a'(x)-a(x)f'(x)\\
 &=& f(x)p'(x)r(x)+f(x)p'(x)r(x)-p(x)r(x)f'(x),
\end{eqnarray*}
which gives $h(x)r'(x)-r(x)h'(x)=0$. Because $\deg(u(x))<\deg(h(x))$, this implies $r(x)=0$ and $a(x)\in \k f(x)$. We have obtained that
$$
Z^{1}_e = \k f \px+\k[x]\theta\pt.
$$
Moreover, for $t_0\in C^0_o=\k[x]$, we have $D(t_0)=\pm f(x)t_0'\theta\pt$, so that $B^{1}_e=\k[x]f\theta\pt$. Finally,
$$
H^{1}_e(\psi) = \k f \px + \k[x]/(f(x))\, \pt,
$$
and $H^n_e(\psi)=$
\begin{equation*}
\begin{cases}
\k f \px + \displaystyle{\frac{\k[x]}{(f(x))}}\, \theta\pt,&n=1,\\
\vspace*{-0.3cm}
{}\\
\left\lbrace rp\px\pt^{n-1} + \displaystyle{\frac{1}{1-n}}(rq-r'p-rp')\theta\pt^n\mid r\in \k[x],\right. \;&\hbox{otherwise.}\\
\qquad\qquad\qquad\qquad\qquad\quad\left.\phantom{\displaystyle{\frac{1}{n-1}}} \deg(r(x))<\deg(h(x))\right\rbrace&
\end{cases}
\end{equation*}
%


%
\section{Poisson structures on a $2|1$-dimensional polynomial algebra}
\label{section21}
Let $\A=\k[x,y,\theta]$, the $2|1$-dimensional polynomial algebra. In this section, we explain what are the conditions for an odd biderivation of $\A$ to be a Poisson structure. We then give explicit families of Poisson structures that satisfy the property of admitting a nontrivial even or odd Casimir. We then finally completly determine the Poisson cohomology associated to one of these families of Poisson structures.
To simplify the notation, we will often denote $\k[x,y]$ by $\A'$.

\begin{rem}

Let us first consider the de Rham complex, associated to the algebra
$\A'=\k[x,y]$. We denote by
$$
\Omega^1(\A')=\lbrace f(x,y)\diff x +g(x,y)\diff y \mid (f, g)\in \A'^2\rbrace
$$
 and by
 $$
 \Omega^2(\A')=\lbrace f(x,y)\diff x\wedge\diff y\mid f\in \A' \rbrace
 $$
  the spaces of $1$ and $2$-(K\"ahler) forms of the algebra $\A'$.
For $(f, g)\in \A'^2$, we have the de Rham differential $\diff$ defined by :
\begin{align*}
\diff f &=f_x\diff x+ f_y\diff y,\\
\diff(f\diff x +g\diff y)&=\diff(f)\wedge\diff x +\diff(g)\wedge\diff y=(g_x-f_y)\diff x\wedge \diff y,\\
\diff(f\diff x\wedge\diff y)&=0,
\end{align*}
where $f_x$ and $f_y$ denote the partial derivatives of $f\in \A'$, with respect to $x$ and $y$.

We identify an element of $\Omega^1(\A')$ with an
element of $\A'^2$ and an element of $\Omega^2(\A')$ with an
element of $\A'$, by the following maps:
$$
\begin{array}{ccc}
\Omega^1(\A') &\ra& \A'^2\\
        f\diff x +g\diff y &\mapsto& \left(f,g\right)
\end{array}\quad
\hbox{ and }\quad
\begin{array}{ccc}
\Omega^2(\A') &\ra& \A'\\
        f\diff x\wedge\diff y &\mapsto& f.
\end{array}
$$
Then, using these identifications, we can write the de Rham complex as
follows:
\begin{equation}\label{eq:dR}
\begin{diagram}
\node{\A'=\k[x,y]}\arrow{e,t}{\vn}
\node{\A'^2}\arrow{e,t}{\Div} \node{\A'}
\end{diagram}
\end{equation}
where the gradient and divergence operators, $\vn$ and $\Div$, are defined as follows:
$$
\begin{array}{ccc}
\vn :\A' &\to& \A'^2\\
        f &\mapsto& \left(\begin{smallmatrix}\displaystyle f_x\\ \displaystyle f_y\end{smallmatrix}\right),
\end{array}
\qquad
\begin{array}{ccc}
\Div :\A'^2 &\to& \A'\\
        \left(\begin{smallmatrix}f\\ g\end{smallmatrix}\right) &\mapsto& g_x- f_y.
\end{array}
$$
Notice that for every $f\in \A'$, we have the following identity :
$$
\Div(\vn(f))=0.
$$
Moreover, the de Rham complex (\ref{eq:dR}) is exact. Indeed, if $f\in \A'$ satisfies $\vn f=\left(\begin{smallmatrix}0\\ 0\end{smallmatrix}\right)$, then of course, $f\in \k$. Next, let us assume that $\left(\begin{smallmatrix}f\\ g\end{smallmatrix}\right)\in \A'^2$ satisfies $\Div\left(\left(\begin{smallmatrix}f\\ g\end{smallmatrix}\right)\right)=0$, this means that $g_x=f_y$. It suffices to show the result for $f$ and $g$, two homogeneous polynomials of the same degree $n\in \N$. Let $h=\frac{1}{n+1}(xf+yg)$. We then have
$$
\vn h =\frac{1}{n+1} \left(\begin{smallmatrix}\displaystyle f + xf_x+yg_x\\ \displaystyle xf_y+yg_y+g\end{smallmatrix}\right)
=\frac{1}{n+1} \left(\begin{smallmatrix}\displaystyle f + xf_x+yf_y\\ \displaystyle xg_x+yg_y+g\end{smallmatrix}\right)
=\left(\begin{smallmatrix}f\\ g\end{smallmatrix}\right),
$$
where we have used Euler's formula $xf_x+yf_y=nf$. This proves that the complex (\ref{eq:dR}) is exact.
\end{rem}
In the following, we will also use the cross product $\times : \A'^2\to \A'$ given, for  $\left(\begin{smallmatrix}\displaystyle f\\ \displaystyle g\end{smallmatrix}\right), \left(\begin{smallmatrix} \displaystyle  h\\ \displaystyle k\end{smallmatrix}\right)\in \A'^2$, by:
$$
\left(\begin{smallmatrix} \displaystyle  f\\ \displaystyle g\end{smallmatrix}\right)\times\left(\begin{smallmatrix} \displaystyle  h\\ \displaystyle k\end{smallmatrix}\right)
=fk-gh.
$$

\smallskip
An odd biderivation $\psi$ must be of the form
\begin{equation*}
\psi=f(x,y)\px\py+g(x,y)\theta\px\pt+h(x,y)\theta\py\pt+k(x,y)\pt^2.
\end{equation*}
We have
\begin{align*}
\tfrac12[\psi,\psi]=&
-(-fg_x+f_xg-fh_y+f_yh)\theta\px\py\pt+
(fk_y-2kg)\px\pt^2
\\&-(fk_x+2hk)\py\pt^2
-(gk_x+hk_y)\theta\pt^3
\end{align*}

The codifferential condition $[\psi,\psi]=0$ is equivalent to the
three conditions
$$
\renewcommand{\arraystretch}{2}
\left\lbrace\begin{array}{l}
f\, \vn k + 2k \left( \begin{smallmatrix}  h\\  -g \end{smallmatrix}\right) = 0,\\
\left( \begin{smallmatrix}  h\\  -g \end{smallmatrix} \right) \times \vn k =0,\\
-\left( \begin{smallmatrix} h\\  -g \end{smallmatrix}\right) \times\vn f - f\,\Div \left( \begin{smallmatrix}  h\\  -g \end{smallmatrix}\right) =0.
\end{array}\right.
$$
It is here easy to see that the second condition follows from the first one (by applying $\times\vn k$ and because $\vn k\times\vn k =0$).

By studying the previous equations, we are able to give a list of different families of Poisson structures that admit Casimirs.
A \emph{Casimir of $\psi$} is a cocycle in $C^0$, in other words, an $\alpha$ element of $\A$
such that $\psi(\alpha,\beta)=0$ for all $\beta\in\A$.
Let us consider the conditions for an even element $\alpha=a(x,y)\theta$
to be a Casimir for the Poisson structure.
\begin{align*}
[\psi,\alpha]=(-fa_y-ga)\theta\px+(fa_x-ah)\theta\py-2ka\pt.
\end{align*}
 It follows that there are nonzero even Casimirs only when $k(x,y)=0$.
Also, a special case of a Poisson structure which has a nontrivial even Casimir $a(x,y)\theta$ is given by the following:
\begin{align*}
\psi &= a(x,y)\px\py-a_y(x,y)\theta\px\pt+a_x(x,y)\theta\py\pt, \hbox{ with } a\in \k[x,y],\\
&(\ie\;\; f=a,\quad  k=0,\quad g=-a_y,\quad h=a_x, \quad a(x,y)\in \k[x,y]).
\end{align*}

On the other hand, suppose that $\beta=b(x,y)$ is an odd element of $C^0$.  Then
\begin{align*}
[\psi,\beta]= -(fb_y\px-fb_x\py-(gb_x+hb_y)\theta\pt).
\end{align*}
Every even constant function is an odd Casimir.
For a nonconstant Casimir $\beta$, it follows that $f(x,y)=0$.
A special case of a Poisson structure with a nonconstant odd Casimir $b(x,y)$ is given by the following:
\begin{align*}
\psi &= b_y(x,y)\theta\px\pt-b_x(x,y)\theta\py\pt, \hbox{ with } b\in \k[x,y],\\
&(\ie\;\; f=k=0,\quad g=b_y,\quad h=-b_x, \quad b(x,y)\in \k[x,y]).
\end{align*}
Another case where there are nontrivial odd Casimirs is given by the following.
\begin{align*}
\psi &= k(x,y)\pt^2, \hbox{ with } k\in \k[x,y],\\
&(\ie\;\; f=g=h=0,\quad  k(x,y)\in\k[x,y]).
\end{align*}
For any such Poisson structure, every function $b(x,y)\in\k[x,y]$ is an odd Casimir.

Another interesting family of Poisson structures is given by the following.
\begin{align*}
\psi &= -2k(x,y)\px\py-k_y(x,y)\theta\px\pt+k_x(x,y)\theta\py\pt+k(x,y)\pt^2,\\
&\hbox{ with } k\in \k[x,y],\\
&(\ie\;\; f =-2k,\qquad g=-k_y,\qquad h=k_x,\qquad k(x,y)\in\k[x,y]).
\end{align*}
When $k\ne0$, the only Casimirs for this type of
Poisson structure are the constant functions $\beta=c\in \k$.

In order to write the Poisson cohomology complex associated to a Poisson structure
$$
\psi =f\px\py + g\theta\px\pt + h\theta\py\pt + k \pt^2,
$$
in terms of the operators $\times$, $\Div$ and $\vn$,
we identify the cochains to elements of the spaces $\A'=\k[x,y]$, $\A'\times \A'^2$ or $\A'\times \A'\times \A'^2$, as follows.
First, the space of odd $0$-cochains $C^0_o = \lbrace c(x,y)\in \A'\rbrace$ is equal to $\A'$, while the space of even $0$-cochains $C^0_e=\lbrace a(x,y)\theta\mid a\in \A'\rbrace$ can be identified to $\A'$, by the following map:
$$
\begin{array}{ccc}
C^0_e &\to &\A'\\
a(x,y)\theta &\mapsto &a(x,y).
\end{array}
$$
Next, we consider the space of odd $1$-cochains $C^1_o=\lbrace p\theta\px + q\theta\py+r\pt\mid (p,q,r)\in \A'^3\rbrace$ and the space of even $1$-cochains $C^1_e=\lbrace a\px+ b\py+ c\theta\pt\mid (a,b,c)\in \A'^3\rbrace$, which will be identified with the space $\A'\times \A'^2$ by the following maps:
$$
\begin{array}{ccc}
C^1_o &\to &\A'\times \A'^2\\
p\theta\px + q\theta\py+r\pt  &\mapsto & \left(r, \left(\begin{smallmatrix}q\\ -p\end{smallmatrix}\right)\right),\\
{}\\
C^1_e &\to &\A'\times \A'^2\\
a\px+ b\py+ c\theta\pt  &\mapsto & \left(c, \left(\begin{smallmatrix}b\\ -a\end{smallmatrix}\right)\right).
\end{array}
$$
Finally, for every $n\geq 2$, the space of odd $n$-cochains $C^n_o=\lbrace p\px\py\pt^{n-2}+ q\theta\px\pt^{n-1}+ r\theta \py\pt^{n-1} + s\pt^n\mid (p, q, r, s)\in \A'^4\rbrace$ and the space of even $n$-cochains $C^n_e=\lbrace a\theta\px\py\pt^{n-2}+ b\px\pt^{n-1}+ c \py\pt^{n-1} + d \theta\pt^n\mid (a,b,c,d)\in \A'^4\rbrace$ will be identified with the space $\A'\times\A'\times \A'^2$, by the following maps:
$$
\begin{array}{ccc}
C^n_o &\to &\A'\times \A'\times \A'^2\\
p\px\py\pt^{n-2}+ q\theta\px\pt^{n-1}+ r\theta \py\pt^{n-1} + s\pt^n  &\mapsto & \left(p, s,  \left(\begin{smallmatrix}r\\ -q\end{smallmatrix}\right)\right),\\
{}\\
C^n_e &\to &\A'\times\A'\times \A'^2\\
a\theta\px\py\pt^{n-2}+ b\px\pt^{n-1}+ c \py\pt^{n-1} + d \theta\pt^n  &\mapsto & \left(a, d, \left(\begin{smallmatrix}c\\ -b\end{smallmatrix}\right)\right).
\end{array}
$$
Also in the following an element $ \left(\begin{smallmatrix}f\\ g\end{smallmatrix}\right)$ of $\A'^2$ will often be denoted by a capital letter with an arrow: $\vec F := \left(\begin{smallmatrix}f\\ g\end{smallmatrix}\right)$. The element $\left(\begin{smallmatrix}0\\ 0\end{smallmatrix}\right)$ in $\A'^2$ will also be denoted by $\vec 0$.

We now want to determine the (odd and even) Poisson cohomology of a Poisson structure on $\A$, of the form
$$
\psi_b := b_y\theta\px\pt - b_x\theta\py\pt,
$$
where $b\in \A'$ is a polynomial.

Let us first point out that in this case, the \emph{singular locus} of the codifferential $\psi_b$ is defined as being the
affine variety
$$
\lbrace b_x=b_y=0\rbrace\subseteq \k^2,
$$
 and because $b$ is supposed to be homogeneous, this singular locus coincide with the singularities of the surface
 $$
 \mathcal F_b := \lbrace (x, y)\in \k^2\mid b(x, y)=0\rbrace\subseteq \k^2.
 $$

From now, we denote by $D_{\psi_b}$ the Poisson coboundary operator $D_{\psi_b}:=\lbrack\psi_b, \, \cdot\rbrack$ and we rather write cochains as elements in $\A'=\k[x,y]$, $\A'\times \A'^2$ or $\A'\times \A'\times \A'^2$, as explained above.
Let us write the values taken by this operator, under these identifications. \\
For $\alpha_0 = c(x,y)\in C^0_o(\A)=\A'$:
$$
D_{\psi_b}(\alpha_0)=-\left(\vn b\times\vn c,\; \vec 0\right)\; \in \A'\times \A'^2\simeq C^1_e(\A);
$$
for $\varphi_0 = a(x,y) \in \A'\simeq C^0_e(\A)$,
$$
D_{\psi_b}(\varphi_0)=\left(0, \; a\vn b\right) \;\in \A'\times\A'^2\simeq C^1_o(\A);
$$
for $\alpha_1 = (r,\vec Q)\in \A'\times \A'^2\simeq C^1_o(\A)$,
$$
D_{\psi_b}(\alpha_1) = \left(-\vn b\times \vec Q,\; \vn b\times\vn r, \; r\vn b \right) \;\in \A'\times\A'\times\A'^2\simeq C^2_e(\A);
$$
for $\varphi_1 = \left(r, \vec Q\right)\in \A'\times \A'^2\simeq C^1_e(\A)$,
$$
D_{\psi_b}(\varphi_1) = \left(0, \; 0, \;\vn(\vec Q\times \vn b)+\Div(\vec Q)\vn b\right) \;\in \A'\times\A'\times \A'^2\simeq C^2_o(\A).
$$
And for all $n\geq 2$, for $\alpha_n = \left(p,s,\vec R\right)\in \A'\times\A'\times\A'^2\simeq C^n_o(\A)$,
\begin{eqnarray*}
D_{\psi_b}(\alpha_n) &=& \left(\vn b\times\vn p -(n-2)\vec R\times\vn b,\; \vn b\times\vn s, \; ns\vn b\right)\\
&\in& \A'\times\A'\times\A'^2\simeq C^{n+1}_e(\A);
\end{eqnarray*}
for $\varphi_n = \left(a,d,\vec C\right)\in \A'\times\A'\times\A'^2\simeq C^n_e(\A)$,
\begin{eqnarray*}
\lefteqn{D_{\psi_b}(\varphi_n) =}\\
&& \left((n-1)\vn b\times \vec C,\; 0,\; (n-1)d\vn b+\vn \left(\vec C\times\vn b\right)+\Div(\vec C)\vn b\right)\\
&&\in \A'\times\A'\times\A'^2\simeq C^{n+1}_o(\A).
\end{eqnarray*}

In order to determine the Poisson cohomology associated to the Poisson structure $\psi_b$, we will assume that $b(x,y)\in \A'$ is a non constant, homogeneous and square-free polynomial. These hypotheses imply in particular that the following Koszul complex is exact :
$$
\setlength{\dgARROWLENGTH}{1cm}
\begin{diagram}\label{eq:Koszul}
\node{0}\arrow{e,t}{}\node{ \A'}\arrow{e,t}{\vn b}\node{\A'^2}\arrow{e,t}{\times\vn b}
\node{\A'}
\end{diagram}
$$
where the first map, from $\A'$ to $\A'^2$, maps an element $a\in \A'$ to $a\vn b$ while the second, from $\A'^2$ to $\A'$, maps $\vec G\in \A'^2$ to $\vec G\times \vn b$.
To prove that the above diagram is exact, we use the fact that, as $b$ is homogeneous, it satisfies Euler's identity : $\deg(b)\,b=xb_x+yb_y$ and because $b$ is non constant and square-free, $b_x$ and $b_y$ are coprime. Now if $\vec G=\left( \begin{matrix} f\\ g \end{matrix}\right)\in \A'^2$ satisfies $\vec G\times\vn b=0$, then one has $f b_y-g b_x=0$. This implies that $b_x$ divides $f$ in $\A'$, so that there exists $a\in \A'$ such that $f = ab_x$. This permits us to conclude that $\vec G=a\vn b$.

Notice moreover that, given a non constant homogeneous polynomial $b\in \k[x,y]$, $b$ is square-free  if and only if  the quotient vector space
$$
\A'_{sing}(b):=\frac{\A'}{\displaystyle( b_x, b_y)}=\frac{\A'}{\displaystyle\lbrace \vn b\times \vec G\mid \vec G\in \A'^2\rbrace}
$$
 is of finite dimension (see \cite{Perrin} for a proof of this fact), and in this case, one says that the surface $\mathcal F_b$ has an \emph{isolated singularity} at the origin.
The algebra $\A'_{sing}(b)$ is called the \emph{Milnor algebra} and its dimension (as a $\k$-vector space) is denoted by $\mu$ and called the \emph{Milnor number} of the singularity of $b$.
\begin{rem}
This Milnor number  and the Milnor algebra $\A'_{sing}(b)$ give information about the singularity of the surface $\mathcal F_b$ (i.e., the singularity of the Poisson structure $\psi_b$), as its multiplicity (see \cite{CoxLittleOShea}). We will see that the algebra $\A'_{sing}(b)$ appear in the Poisson cohomology spaces associated to $\psi_b$, so that this Poisson cohomology is linked to the type of the singularity of $\psi_b$. These results have to be compared with analogous results obtained in \cite{Pichereau1}, where the studied Poisson structures are non graded Poisson structures, in dimension three (i.e., defined on $\C[x,y,z]$), of the form:
$\varphi_z\, \partial_x\wedge\partial_y+ \varphi_y\, \partial_z\wedge\partial_x+\varphi_x\, \partial_y\wedge\partial_z$, where $\varphi\in \C[x,y,z]$ is a (weight-)homogeneous polynomial with an isolated singularity at the origin.
\end{rem}
 We denote by $u_0=1, u_1, \dots, u_{\mu-1}\in \A'$ homogeneous polynomials in $\A'$ such that their images in the quotient $\frac{\A'}{\displaystyle( b_x, b_y)}$ give a $\k$-basis of this quotient vector space. One can then write :
\begin{equation}\label{eq:finitedim}
\A'=\k u_0\oplus \k u_1\oplus \cdots\oplus \k u_{\mu-1}\oplus \lbrace \vn b\times \vec G\mid \vec G\in \A'^2\rbrace.
\end{equation}
Note that in the particular case where the degree of $b$ is $1$, then $\mu=0$ and stricly speaking, if we demand that $\mathcal F_b$ \emph{has} a singularity at the origin, we should probably suppose that the degree of $b$ is greater or equal to $2$.
%

%
\begin{prp}\label{prp:H0o}
Let $b(x,y)\in \A'$ be a non-constant homogeneous polynomial. Let $\psi_b$ be the Poisson structure given by the following formula :
$$
\psi_b = b_y\theta\px\pt - b_x\theta\py\pt.
$$
If $b$ is square-free then a basis of the odd Poisson cohomology $0$-space is given by the following :
$$
H^0_o(\A, \psi_b)= \k[b].
$$
\end{prp}
\begin{proof}
First, recall that one can write, under the identifications given above:
$$
H^0_o(\psi_b)=\lbrace c(x,y)\in \A'\mid \vn c\times\vn b=0\rbrace.
$$
Then let $c\in \A'$ such that $\vn c\times\vn b=0$. Because of the exactness of the Koszul complex, there exists $a(x,y)\in툎A'$ such that $\vn c=a\vn b$.  Assume that $c$ is a homogeneous polynomial then, using Euler's formula we obtain:
$$
\deg(c)\, c= xc_x+yc_y = a(xb_x+yb_y)=\deg(b) ab.
$$
This implies that either $\deg(c)=0$ or $b$ divides the polynomial $c$ in $\A'$. We then write $c=b^r h$ with $r\in \N$ and $h\in \A'$, with $b$ non dividing the polynomial $h$ in $\A'$. Then,
$$
\vn c=rb^{r-1}h\vn b + b^r\vn h \; \hbox{ and } \; 0=\vn c\times\vn b = b^r\vn h\times \vn b.
$$
From the above, we obtain that $\deg(h)=0$ and $c\in \k b^r$.
\end{proof}
\begin{prp}
Let $b(x,y)\in \A'$ be a non-constant homogeneous polynomial. Let $\psi_b$ be the Poisson structure given by the formula :
$$
\psi_b = b_y\theta\px\pt - b_x\theta\py\pt.
$$
If $b$ is square-free then the odd Poisson cohomology $1$-space vanishes:
$$
H^1_o(\A, \psi_b)\simeq \lbrace 0\rbrace.
$$
\end{prp}
\begin{proof}
First, we have :
$$
H^1_o(\psi_b)\simeq \frac{\displaystyle \lbrace(r,\vec Q)\in \A'\times\A'^2\mid \vn b\times \vec Q=0; r\vn b=\vec 0; \vn b\times\vn r=0\rbrace}{\displaystyle \lbrace (0, a\vn b)\in \A'\times\A'^2\mid a\in \A'\rbrace}.
$$
Let $(r,\vec Q)\in Z^1_o(\psi_b)$. Then, because $b$ is non-constant and because $ r\vn b=\vec 0$, we have that $r=0$. Moreover, using the exactness of the Koszul complex the condition $ \vn b\times \vec Q=0$ implies that there exists $a\in \A'$ such that $\vec Q=a\vn b$, which permits to conclude that $(r,\vec Q)\in B^1_o(\psi_b)$ and that $H^1_o(\A, \psi_b)\simeq \lbrace 0\rbrace$.
\end{proof}
\begin{prp}
Let $b(x,y)\in \A'$ be a non-constant homogeneous polynomial. Let $n\in \N$ satisfying $n\geq 3$. Let $\psi_b$ be the Poisson structure given by the following formula :
$$
\psi_b = b_y\theta\px\pt - b_x\theta\py\pt.
$$
If $b$ is square-free then a basis of the odd Poisson cohomology $n$-space is given by the following :
\begin{eqnarray*}
H^n_o(\psi_b)&\simeq& \bigoplus_{i=0}^{\mu-1}\k\left((n-2)u_i, 0, -\vn u_i \right)\\
                             &\simeq& \bigoplus_{i=0}^{\mu-1}\k\left((n-2)u_i\px\py\pt^{n-2}+(u_i)_y\theta\px\pt^{n-1}-(u_i)_x\theta\py\pt^{n-1}\right).
\end{eqnarray*}
\end{prp}
\begin{rem}
With this result, one easily sees that $H^n_o(\psi_b)\simeq \A'_{sing}(b)$, when $n\geq 3$.
\end{rem}
\begin{proof}
Let us recall that one can write:
\begin{eqnarray*}
\lefteqn{H^n_o(\psi_b) \simeq}\\
&&\!\!\frac{\displaystyle\left\lbrace\begin{array}{l}  (p, s, \vec R)\in \A'\times\A'\times\A'^2\mid\\
\phantom{blabla} \vn b\times\vn p+(n-2)\vn b\times\vec R=0; ns\vn b=\vec 0; \vn b\times\vn s=0\end{array}\right\rbrace}{\displaystyle\left\lbrace\begin{array}{l} \left((n-2)\vn b\times \vec C, 0, \vn\left(\vec C\times\vn b\right)+\Div(\vec C)\vn b+(n-2)d\vn b\right)\\
\qquad \qquad \qquad \qquad \qquad \qquad\mid (a, d, \vec C)\in \A'\times\A'\times\A'^2\end{array}\right\rbrace}.
\end{eqnarray*}
Let us then consider an element $(p, s, \vec R)\in Z^n_o(\psi_b)$. As $ns\vn b=\vec 0$ and $b$ is an non-constant polynomial, one necessarily obtains that $s=0$. The cocycle condition now becomes
$$
 \vn b\times\vn p+(n-2)\vn b\times\vec R=0, \; \hbox{ i.e., }\;
\vn b\times\left(  \vn p+(n-2)\vec R\right)=0.
$$
Because the Koszul complex is exact, this implies the existence of an element $f\in \A'$ satisfying
\begin{equation}\label{eq:eqncocy}
 \vn p+(n-2)\vec R = (n-2) f\vn b.
 \end{equation}
Now, (\ref{eq:finitedim}) implies that there exist $\lambda_0, \lambda_1, \dots, \lambda_{\mu-1}\in\k$  and $\vec C\in \A'^2$ such that :
$$
p=(n-2)\vn b\times \vec C + \sum_{i=0}^{\mu-1}\lambda_i (n-2)u_i.
$$
Thus,
$$
\vn p=(n-2)\vn\left(\vn b\times \vec C\right) + \sum_{i=0}^{\mu-1}\lambda_i (n-2)\vn u_i,
$$
and (\ref{eq:eqncocy}) becomes:
$$
\vec R=-\vn\left(\vn b\times \vec C\right) - \sum_{i=0}^{\mu-1}\lambda_i \vn u_i + f\vn b.
$$
Let $d:=\frac{1}{(n-2)}\left(f-\Div(\vec C)\right)$, then:
$$
\vec R=-\vn\left(\vn b\times \vec C\right) + (n-2)d\vn b + \Div(\vec C)\vn b- \sum_{i=0}^{\mu-1}\lambda_i \vn u_i.
$$
Finally,
\begin{eqnarray*}
(p, s, \vec R) &=& \left((n-2)\vn b\times \vec C, 0,  -\vn\left(\vn b\times \vec C\right) + (n-2)d\vn b + \Div(\vec C)\vn b\right) \\
      &+&\sum_{i=0}^{\mu-1}\lambda_i\left((n-2)u_i, 0, -\vn u_i\right)\\
      &\in& B^n_o(\psi_b) + \sum_{i=0}^{\mu-1}\k\left((n-2)u_i, 0, -\vn u_i\right).
\end{eqnarray*}
which implies that:
$$
H^n_o(\psi_b)= \sum_{i=0}^{\mu-1}\k \left((n-2)u_i, 0, -\vn u_i\right).
$$
It now remains to show that this sum is a direct one, by considering $\lambda_0, \lambda_1, \dots, \lambda_{\mu-1}\in \k$ and $(d, \vec C)\in\A'\times\A'^2$ such that
\begin{eqnarray*}
\lefteqn{\sum_{i=0}^{\mu-1}\lambda_i\left((n-2)u_i, 0, -\vn(u_i)\right)=}\\
&& \left((n-2)\vn b\times \vec C, 0, -\vn\left(\vn b\times C\right)+\Div(\vec C)\vn b+(n-2)d\vn b\right).
\end{eqnarray*}
But then, $\sum_{i=0}^{\mu-1}\lambda_i u_i=\vn b\times \vec C\in \langle b_x, b_y\rangle$, so that, by definition of the $u_i$, we conclude that $\lambda_i=0$ for all $i=0, \dots, \mu-1$. We finally have obtained the desired result.
\end{proof}

The difficult part of the computation of the odd Poisson cohomology associated to the Poisson structure $\psi_b= b_y\theta\px\pt - b_x\theta\py\pt$ lies in the second Poisson cohomology space, which we give here.
\begin{prp}\label{prp:H2}
Let $b(x,y)\in \A'$ be a non-constant homogeneous polynomial. Let $\psi_b$ be the Poisson structure given above.
If $b$ is square-free then a basis of the odd Poisson cohomology $2$nd-space is given by:
\begin{eqnarray*}
H^2_o(\psi_b)&\simeq& \k[b]\; (1, 0, 0)  \oplus \bigoplus_{\stackrel{i=0, \dots, \mu-1}{\deg(u_i)=\deg(b)-2}}
\k[b]\; u_i\;(0, 0, \vn b)\\
&\oplus& \bigoplus_{j=0}^{\mu-1} \k[b]\; u_j\;(0, 0, \vec E)
\;\;\oplus\;\;\bigoplus_{\stackrel{i=0, \dots, \mu-1}{\deg(u_i)\not=\deg(b)-2}}\k[b]\;(0, 0, \vn u_i)\\
&\oplus&
\bigoplus_{\stackrel{j=1, \dots, \mu-1}{\deg(u_j)=\deg(b)-2}}\k\; (0, 0, \vn u_j),
\end{eqnarray*}
where $\vec E:=\left(\begin{smallmatrix}y\\ -x\end{smallmatrix}\right)\in \A'^2$.
This can be written more explicitly as
\begin{eqnarray*}
\lefteqn{H^2_o(\psi_b)\;\;\simeq\;\;\k[b]\; \px\py \;\oplus \bigoplus_{\stackrel{i=0, \dots, \mu-1}{\deg(u_i)=\deg(b)-2}}\k[b]\;  u_i\; \psi_b}\\
&\oplus&\bigoplus_{k=0}^{\mu-1} \k[b]\; u_j\;\left(x\theta\px\pt+y\theta\py\pt\right)\\
&\oplus&\bigoplus_{\stackrel{i=0, \dots, \mu-1}{\deg(u_i)\not=\deg(b)-2}}\k[b]\;\left((u_i)_y\theta\px\pt-(u_i)_x\theta\py\pt\right) \\
&\oplus&\bigoplus_{\stackrel{j=1, \dots, \mu-1}{\deg(u_j)=\deg(b)-2}}\k\; \left((u_j)_y\theta\px\pt-(u_j)_x\theta\py\pt\right).
\end{eqnarray*}
\end{prp}
In order to be able to prove this proposition, we need the following
\begin{lma}\label{lma:sqfree2}
Let $b(x,y)\in \A'$ be a non-constant homogeneous polynomial. If $b$ is square-free then, we have:
\begin{equation}\label{eq:sqfree2}
\A'=\bigoplus_{i=0}^{\mu-1}\k[b] u_i \;\oplus\; \lbrace \vn h\times \vn b\mid h\in \A'\rbrace.
\end{equation}
\end{lma}
\begin{proof}[Proof of Lemma \ref{lma:sqfree2}]
We first prove that $\A'=\sum_{i=0}^{\mu-1}\k[b] u_i \;\oplus\; \lbrace \vn h\times \vn b\mid h\in \A'\rbrace$. To do this, let $f\in \A'$ a homogeneous polynomial in $\k[x,y]$. According to (\ref{eq:finitedim}), there exist $\lambda_0, \lambda_1, \dots, \lambda_{\mu-1}\in \k$ and $\vec F=\left(\begin{smallmatrix}f_1\\ f_2\end{smallmatrix}\right)\in \A'^2$ such that
\begin{equation}\label{eq:F}
f= \vec F\times \vn b+ \sum_{i=0}^{\mu-1}\lambda_i u_i,
\end{equation}
and such that $f_1$ and $f_2$ are two homogeneous polynomials of $\k[x,y]$, of degree $\deg(f_1)=\deg(f_2)=\deg(f)-\deg(b)+1$. We will now proceed by induction on $\deg(f)$.

First, if $\deg(f)\leq \deg(b)-1$, then $f_1=a\in \k$ and $f_2=b\in \k$, so that one can write $\vec F=\vn h$, with $h:= ax+by\in \A'$ and $f=\vn h\times \vn b+\sum_{i=0}^{\mu-1}\lambda_iu_i$.

Secondly, let $d=\deg(f)$ and suppose that $d\geq \deg(b)$. We also suppose that, for any homogeneous polynomial $g\in A'$ of degree less or equal to $d-1$, we have $g\in \sum_{i=0}^{\mu-1}\k[b] u_i \;\oplus\; \lbrace \vn h\times \vn b\mid h\in \A'\rbrace$. Because of Euler's formula, for any homogeneous polyomial $k\in툎A'$, we have $\Div(k\vec E)=-(\deg(k)+2)k$. As $\Div(\vec F)=0$  or $\deg(\Div(\vec F))=\deg(f)-\deg(b)$, we have:
$$
\Div(\Div(\vec F)\vec E + (\deg(f)-\deg(b)+2)\vec F)=0.
$$
Now, using the exactness of the de Rham complex, we obtain the existence of a homogeneous polynomial $k\in \A'$ such that:
\begin{eqnarray}\label{eq:vecF}
 \vec F = \frac{-1}{(\deg(f)-\deg(b)+2)} \Div(\vec F)\vec E -\vn k.
\end{eqnarray}
Moreover, $\deg(\Div(\vec F))=\deg(f)-\deg(b)<\deg(f)$ (by hypothesis, $b$ is non-constant), so that we can apply the induction hypothesis on $\Div(\vec F)$ and obtain the existence of $\ell\in \A'$ and $\alpha_{i,j}\in \k$, for all $j\in \N$ and $i=0,\dots, \mu-1$ such that:
$$
\Div(\vec F)=\vn \ell\times \vn b+\sum_{i=0}^{\mu-1}\sum_{j\in \N}\alpha_{i,j}\; b^j\; u_i,
$$
where of course, for each $i\in\lbrace 0, \dots, \mu-1\rbrace$, only a finite number of the $\alpha_{i,j}$ are non-zero (so that the previous sum is well-defined). Then, by (\ref{eq:vecF}), we have:
$$
\vec F= \frac{-1}{(\deg(f)-\deg(b)+2)} \left(\vn \ell\times \vn b+\sum_{i=0}^{\mu-1}\sum_{j\in \N}\alpha_{i,j}\; b^j\; u_i\right)\vec E -\vn k,
$$
and by (\ref{eq:F}), we obtain:
\begin{eqnarray*}
f &=&  \frac{-1}{(\deg(f)-\deg(b)+2)} \left(\vn \ell\times \vn b+\sum_{i=0}^{\mu-1}\sum_{j\in \N}\alpha_{i,j}\; b^j\; u_i\right)\vec E\times \vn b\\
&& - \vn k\times\vn b\; +\; \sum_{i=0}^{\mu-1}\lambda_i u_i.
\end{eqnarray*}
Now, by Euler's formula, we compute $\vec E\times \vn b=\deg(b)\, b$ and because $\vn b\times\vn b=0$, we also write $ \left(\vn \ell\times \vn b\right)b=\left(\vn (\ell b)\times \vn b\right)$. Thus,
\begin{eqnarray*}
f &=&  \frac{-\deg(b)}{(\deg(f)-\deg(b)+2)} \left(\vn (b\ell)\times \vn b\right)\\
&&+ \sum_{i=0}^{\mu-1}\sum_{j\in \N} \left(\frac{\deg(b) \;\alpha_{i,j}}{(\deg(f)-\deg(b)+2)}\right)\; b^{j+1}\; u_i\\
&& - \vn k\times\vn b\; +\; \sum_{i=0}^{\mu-1}\lambda_i u_i\\
&\in & \lbrace\vn h\times\vn b\mid h\in \A' \rbrace + \sum_{i=0}^{\mu-1}\k[b]\; u_i.
\end{eqnarray*}
We then have shown that $\A'=\sum_{i=0}^{\mu-1}\k[b] u_i \;+\; \lbrace \vn h\times \vn b\mid h\in \A'\rbrace$,
and it remains to show that this sum is a direct one. To do this, we suppose on the contrary that this sum is not direct.
%
%
%
Then we define $j_0$ as being the smaller integer such that there exists $0\leq i_0\leq \mu-1$, a family of constants $\gamma_{i,j}\in \k$, where $j\in \N$, $i=0,\dots, \mu-1$ and $\gamma_{i_0, j_0}\not=0$ and $p\in \A'$ satisfying an equation of the form:
\begin{eqnarray}\label{eq:directsumH2}
\sum_{i=0}^{\mu-1}\sum_{j\in \N} \gamma_{i,j}\; b^j\; u_i=\vn p\times\vn b.
\end{eqnarray}
Now, if $j_0=0$, then there exist a family of constants $\gamma_{i,j}\in \k$, where $j\in \N$, $i=0,\dots, \mu-1$ and $\gamma_{i_0, 0}\not=0$ and $p\in \A'$ satisfying:
\begin{eqnarray*}
\sum_{i=0}^{\mu-1} \gamma_{i, 0}\; u_i&=&
-\sum_{i=0}^{\mu-1}\sum_{j\in \N^*} \gamma_{i,j}\; b^j\; u_i +\vn p\times\vn b.
\end{eqnarray*}
As $b=\deg(b)(xb_x+yb_y)$, this leads to:
$$
\sum_{i=0}^{\mu-1} \gamma_{i, 0}\; u_i \in \langle b_x, b_y\rangle,
$$
which implies, regarding the definition of the $u_i$, that $\gamma_{i, 0}=0$, for all $0\leq i\leq \mu-1$. We obtain a contradiction with the definition of $j_0$.

Now, assuming that $j_0\geq 1$ and using once more $b=\frac{1}{\deg(b)}\vec E\times \vn b$ in~(\ref{eq:directsumH2}),
$$
\sum_{i=0}^{\mu-1}\sum_{j\geq j_0} \frac{\gamma_{i,j}}{\deg(b)}\; b^{j-1}\; u_i\;\vec E\times \vn b=\vn p\times\vn b.
$$
(Recall that, according to the definition of $j_0$, for all $0\leq i\leq \mu-1$ and all $j\leq j_0-1$, one has $\gamma_{i,j}=0$.)
As the Koszul complex is exact, there exists $d\in \A'$ satisfying:
$$
\sum_{i=0}^{\mu-1}\sum_{j\geq j_0} \frac{\gamma_{i,j}}{\deg(b)}\; b^{j-1}\; u_i\vec E=\vn p+d\vn b.
$$
Computing the divergence of this,
$$
\sum_{i=0}^{\mu-1}\sum_{j\geq j_0} \frac{\gamma_{i,j}}{\deg(b)}\left(\deg(b)(j-1)+\deg(u_i)+2\right)\; b^{j-1}\; u_i= - \vn d\times\vn b.
$$
Denoting by $\tilde \gamma_{i,j}$ the constant : $\tilde \gamma_{i,j}:=\frac{\gamma_{i,j+1}}{\deg(b)}\left(\deg(b)j+\deg(u_i)+2\right)$, we obtain the equation:
$$
\sum_{i=0}^{\mu-1}\sum_{j'\geq j_0-1} \tilde\gamma_{i,j'}\; b^{j'}\; u_i=- \vn d\times\vn b,
$$
with $\tilde\gamma_{i_0, j_0-1}=\frac{\gamma_{i_0, j_0}}{\deg(b)}\left(\deg(b)(j_0-1)+\deg(u_{i_0})+2\right)\not= 0$. We obtain a contradiction with the definition of $j_0$.
Finally, we have shown the fact that the previous sum is direct and the lemma is proved.
\end{proof}

We now prove Proposition \ref{prp:H2}. To do this, we denote by $D':\A'^2\to \A'^2$ the operator given, for $\vec Q\in \A'^2$ by $D'(\vec Q):= -\vn \left(\vn b\times\vec Q\right)+\Div(\vec Q)\, \vn b$.
\begin{rem}\label{rem:D'(HE)}
Using Euler's formula, we obtain, for every homogeneous polynomial $h\in \A'$,
\begin{equation}\label{eq:D'(HE)1}
\D'(h\vec E) = \deg(b)\,b\, \vn h+\left(\deg(b)-\deg(h)-2\right)h\vn b.
\end{equation}
Secondly, we compute
$$
\left(\vn h\times \vn b\right)\vec E = \deg(b)\, b\, \vn h-\deg(h)\, h\, \vn b.
$$
These two equalities imply:
\begin{equation}\label{eq:D'(HE)}
\D'(h\vec E) = \left(\vn h\times\vn b\right)\vec E + (\deg(b)-2)h\vn b.
\end{equation}
\end{rem}
\begin{rem}
According to proposition \ref{prp:H2}, we see that the Poisson structure $\psi_b$ is an odd $2$-coboundary for the Poisson cohomology associated to $\psi_b$ itself if and only if $\deg(b)\not=2$ and this is due to the equality $D'(\vec E)=\left(\deg(b)-2\right)\vn b$.
\end{rem}
{\bf Proof of proposition \ref{prp:H2}}\\
Recall that, with the help of the identifications given previously, we have
\begin{eqnarray*}
\lefteqn{H^2_o(\psi_b) \simeq }\\
&&\!\!\frac{\displaystyle\left\lbrace(p, s, \vec R)\in \A'\times\A'\times\A'^2\mid \vn b\times\vn p=0; 2s\vn b=\vec 0; \vn b\times\vn s=0 \right\rbrace}{\displaystyle\left\lbrace\left(0, 0, D'(\vec C)=-\vn\left(\vn b\times\vec C\right)+\Div(\vec C)\vn b\right)\mid \vec C\in \A'^2\right\rbrace}.
\end{eqnarray*}
Let $(p, s, \vec R)\in Z^2_o(\psi_b)$. As $s\, \vn b=0$, we have $s=0$. Moreover, the equation $\vn b\times\vn p=0$ and proposition \ref{prp:H0o} imply that $p\in \k[b]$.
This shows that one can write
\begin{eqnarray*}
H^2_o(\psi_b) &\simeq& \k[b]\px\py\oplus \frac{\displaystyle\left\lbrace \vec R\in \A'^2 \right\rbrace}{\displaystyle\left\lbrace-\vn\left(\vn b\times\vec C\right)+\Div(\vec C)\vn b\mid \vec C\in \A'^2\right\rbrace}.
\end{eqnarray*}
Let $\vec F\in \A'^2$ be a homogeneous element, that is $\vec F=\left(\begin{smallmatrix}f_1\\ f_2\end{smallmatrix}\right)$ such that $f_1$ and $f_2$ are homogeneous polynomials of the same degree $\deg(f_1)=\deg(f_2)=d\in \N$.

Because $\Div\left(\Div(\vec F)\vec E\right)=-(d+1)\Div(\vec F)$ and because the de Rham complex is exact, there exists $k\in \A'$ such that
$$
\vec F=\frac{-1}{d+1}\Div(\vec F)\vec E+\vn k.
$$
 According to lemma \ref{lma:sqfree2}, there exist $h\in \A'$ and $\lambda_{i,j}\in \k$, for $i\in \N$ and $0\leq j\leq \mu-1$, satisfying
$$
k = \vn h\times \vn b+\sum_{j=0}^{\mu-1}\sum_{i\in \N}\lambda_{i,j}\; b^i\; u_j,
$$
where for each $0\leq j\leq \mu-1$, only a finite number of the $\lambda_{i,j}$ are non-zero.
Then,
\begin{eqnarray*}
\vn k &=&\vn\left( \vn h\times \vn b\right)+\sum_{j=1}^{\mu-1}\sum_{i\in \N}\lambda_{i,j}\; b^i\; \vn u_j
+\sum_{j=0}^{\mu-1}\sum_{i\in \N^*}\lambda_{i,j}\;i\; b^{i-1}\; u_j\;\vn b,\\
   &=& D'(\vn h)+\sum_{j=1}^{\mu-1}\sum_{i\in \N}\lambda_{i,j}\;  b^i\; \vn u_j
+\sum_{j=0}^{\mu-1}\sum_{i\in \N^*}\lambda_{i,j}\;i\; b^{i-1}\; u_j\; \vn b.
\end{eqnarray*}
Moreover, applying successively two times lemma \ref{lma:sqfree2},  we obtain the existence of $h', \ell\in\A'$ and some constants $\alpha_{i,j}, \gamma_{i,j}\in \k$, for $i\in \N$ and $0\leq j\leq \mu-1$ (with, for each $j$, only a finite number of non-zero $\alpha_{i,j}$ and $\gamma_{i,j}$) such that:
\begin{eqnarray*}
\Div(\vec F) &=& \vn h'\times\vn b + \sum_{j=0}^{\mu-1}\sum_{i\in \N}\alpha_{i,j}\;b^i\; u_j,\\
h' &=& \vn \ell\times\vn b + \sum_{j=0}^{\mu-1}\sum_{i\in \N}\gamma_{i,j}\;b^i\; u_j.
\end{eqnarray*}
We compute $D'\left(\ell\vn b\right)=\Div\left(\ell\vn b\right)\vn b=\left(\vn \ell\times\vn b\right)\vn b$, so that
$$
h'\vn b = D'\left(\ell\vn b\right) + \sum_{j=0}^{\mu-1}\sum_{i\in \N}\gamma_{i,j}\; b^i\; u_j\; \vn b.
$$
According to (\ref{eq:D'(HE)}), $\left(\vn h'\times \vn b\right)\vec E = D'\left(h'\vec E\right) - (\deg(b)-2) h'\vn b$.
This permits us to write:
\begin{eqnarray*}
\Div(\vec F)\vec E &=&D'\left(h'\vec E\right) - (\deg(b)-2)  D'\left(\ell\vn b\right) \\
&&- \sum_{j=0}^{\mu-1}\sum_{i\in \N} (\deg(b)-2)\gamma_{i,j}\;b^i\; u_j\; \vn b\\
&& + \sum_{j=0}^{\mu-1}\sum_{i\in \N}\alpha_{i,j}\; b^i \;u_j\;\vec E.
\end{eqnarray*}
We finally obtain:
\begin{eqnarray*}
\vec F &=& \frac{-1}{d+1}\Div(\vec F)\vec E+\vn k\\
           &=&  D'\left(\frac{-1}{d+1}h'\vec E+\frac{(\deg(b)-2)}{d+1}\ell\vn b+\vn h\right) \\
           &+& \sum_{j=0}^{\mu-1}\sum_{i\in \N}  \frac{(\deg(b)-2)}{d+1}\gamma_{i,j}\;b^i\; u_j\; \vn b
       + \sum_{j=0}^{\mu-1}\sum_{i\in \N} \frac{-1}{d+1}\alpha_{i,j}\;b^i\; u_j\;\vec E \\
      &+&\sum_{j=1}^{\mu-1}\sum_{i\in \N}\lambda_{i,j}\;b^i\; \vn u_j
+\sum_{j=0}^{\mu-1}\sum_{i\in \N^*}\lambda_{i,j}\;i \;b^{i-1}\; u_j\;\vn b\\
&\in & \im(D') +\sum_{j=0}^{\mu-1}\k[b]u_j\, \vn b + \sum_{j=0}^{\mu-1}\k[b] \;u_j\;\vec E
+\sum_{j=1}^{\mu-1}\k[b]\; \vn u_j.
\end{eqnarray*}
Now, using equation (\ref{eq:D'(HE)1}) (with $h=u_j$) and the fact that $D'(b\vec Q)=bD'(\vec Q)$, for all $\vec Q\in \A'^2$, we obtain
$$
b^i\vn u_j=D'\left(\frac{b^{i-1}}{\deg(b)}u_j\vec E\right),\quad \hbox{ if } i\geq 1 \hbox{ and if } \deg(u_j)=\deg(b)-2,
$$
hence
\begin{eqnarray*}
\sum_{j=1}^{\mu-1}\k[b]\; \vn u_j &\in& \im(D')
\;+ \sum_{\stackrel{1\leq j\leq \mu-1}{\deg(u_j)=\deg(b)-2}}\k\; \vn u_j\\
&&+ \sum_{\stackrel{1\leq j\leq \mu-1}{\deg(u_j)\not=\deg(b)-2}}\k[b]\;\vn u_j.
\end{eqnarray*}
Moreover, if $\deg(u_j)\not=\deg(b)-2$, (\ref{eq:D'(HE)1}) implies
$$
u_j\vn b = \frac{1}{\deg(b)-\deg(u_j)-2}D'\left(u_j\vec E\right)-  \frac{\deg(b)}{\deg(b)-\deg(u_j)-2}b\vn u_j,
$$
hence,
\begin{eqnarray*}
\sum_{j=0}^{\mu-1}\k[b]\; u_j\; \vn b &\in &\im(D')
\;+ \sum_{\stackrel{0\leq j\leq \mu-1}{\deg(u_j)=\deg(b)-2}}\k[b]\; u_j\;\vn b\\
&&+ \sum_{\stackrel{1\leq j\leq \mu-1}{\deg(u_j)\not=\deg(b)-2}}\k[b]\;\vn u_j.
\end{eqnarray*}
This leads to:
\begin{eqnarray*}
\vec F &\in &\im(D')
\;+ \sum_{\stackrel{0\leq j\leq \mu-1}{\deg(u_j)=\deg(b)-2}}\k[b]\; u_j\;\vn b\\
&&+ \sum_{\stackrel{1\leq j\leq \mu-1}{\deg(u_j)\not=\deg(b)-2}}\k[b]\; \vn u_j
\;\;+\sum_{0\leq j\leq \mu-1}\k[b] \; u_j\;\vec E\\
&& + \sum_{\stackrel{1\leq j\leq \mu-1}{\deg(u_j)=\deg(b)-2}}\k\; \vn u_j
\end{eqnarray*}
and
\begin{eqnarray*}
\lefteqn{\frac{\A'^2}{\displaystyle\left\lbrace D'(\vec C)\mid \vec C\in \A'^2\right\rbrace}\simeq \sum_{\stackrel{0\leq j\leq \mu-1}{\deg(u_j)=\deg(b)-2}}\k[b]\; u_j\;\vn b}\\
&&\qquad \qquad + \sum_{\stackrel{1\leq j\leq \mu-1}{\deg(u_j)\not=\deg(b)-2}}\k[b]\;\vn u_j
 +\sum_{0\leq j\leq \mu-1}\k[b] \; u_j\;\vec E\\
 &&\qquad \qquad + \sum_{\stackrel{1\leq j\leq \mu-1}{\deg(u_j)=\deg(b)-2}}\k\; \vn(u_j).
\end{eqnarray*}
If $\deg(b)=1$, then $\mu=0$ and this vector space is $\{0\}$.
It  now remains to show that this sum is a direct one, in the case $\deg(b)\geq 2$. To do this, let us suppose that there exist some constants $e_{k,\ell}, m_{s,t}, c_{r,j}, a_i\in \k$, for $0\leq \ell, j \leq \mu-1$, $1\leq t, i\leq \mu-1$, $k, s, r\in \N$, and there exists $\vec Q\in \A'^2$, satisfying:
\begin{eqnarray}\label{eq:directsumH2c}
\lefteqn{\Div(\vec Q)\vn b-\vn \left(\vn b\times \vec Q\right)=}\nonumber\\
&&
\sum_{\stackrel{k\in \N, 0\leq \ell\leq \mu-1}{\deg(u_\ell)=\deg(b)-2}}e_{k,\ell} \;b^k\; u_\ell\; \vn b
+ \sum_{\stackrel{s\in \N, 1\leq t\leq \mu-1}{\deg(u_t)\not=\deg(b)-2}} m_{s,t}\; b^s\; \vn u_t\\
&&+ \sum_{r\in \N, 0\leq j\leq \mu-1} c_{r,j}\; b^r\; u_j\;\vec E
+ \sum_{\stackrel{1\leq i\leq \mu-1}{\deg(u_i)=\deg(b)-2}} a_i\; \vn u_i,\nonumber
\end{eqnarray}
(where, as usual, all the sums are supposed to be finite). By computing the divergence of these terms and because, for every $k, \ell\in \A'$, one has $\Div(k\vn \ell)=\vn k\times\vn \ell$, one obtains that:
$$
 \sum_{r\in \N, 0\leq j\leq \mu-1} c_{r,j}\left(r\deg(b)+\deg(u_j)+2\right)\; b^r\; u_j \in\lbrace \vn h\times\vn b\mid h\in \A'\rbrace.
$$
Together with (\ref{eq:sqfree2}), this implies that $c_{r,j}=0$, for all $r,j$.

%
%
Now, by computing the cross product of (\ref{eq:directsumH2c}) and $\vec E$, we obtain that:
$$
\sum_{\stackrel{1\leq t\leq \mu-1}{\deg(u_t)\not=\deg(b)-2}} m_{0,t}\deg(u_t)\;  u_t
+\sum_{\stackrel{1\leq i\leq \mu-1}{\deg(u_i)=\deg(b)-2}}\deg(u_i) a_i\; u_i \in \langle b_x, b_y\rangle,
$$
so that, by definition of the $u_i$, this equality implies that $m_{0, t}=0$ and $a_i=0$, for all $t$ and $i$. It now remains:
\begin{eqnarray*}
\lefteqn{\Div(\vec Q)\vn b-\vn \left(\vn b\times \vec Q\right)}\\
&=&\sum_{\stackrel{k\in \N, 0\leq \ell\leq \mu-1}{\deg(u_\ell)=\deg(b)-2}}e_{k,\ell} \;b^k\; u_\ell\; \vn b
+ \sum_{\stackrel{s\in \N^*, 1\leq t\leq \mu-1}{\deg(u_t)\not=\deg(b)-2}} m_{s,t}\; b^s\; \vn u_t\\
&=&\sum_{\stackrel{k\in \N, 0\leq \ell\leq \mu-1}{\deg(u_\ell)=\deg(b)-2}}e_{k,\ell} \;b^k\; u_\ell\; \vn b
+ \sum_{\stackrel{s\in \N^*, 1\leq t\leq \mu-1}{\deg(u_t)\not=\deg(b)-2}} m_{s,t}\; \vn\left(b^s\;  u_t\right)\\
&&- \sum_{\stackrel{s\in \N^*, 1\leq t\leq \mu-1}{\deg(u_t)\not=\deg(b)-2}}s\, m_{s,t}\; u_t\; b^{s-1}\; \vn b.
\end{eqnarray*}
This implies
\begin{eqnarray*}
\lefteqn{\vn\left(\sum_{\stackrel{s\in \N^*, 1\leq t\leq \mu-1}{\deg(u_t)\not=\deg(b)-2}} m_{s,t}\; b^s\;  u_t
+\left(\vn b\times \vec Q\right)\right)}\\
&=&\left(-\sum_{\stackrel{k\in \N, 0\leq \ell\leq \mu-1}{\deg(u_\ell)=\deg(b)-2}}e_{k,\ell} \;b^k\; u_\ell
+\sum_{\stackrel{s\in \N^*, 1\leq t\leq \mu-1}{\deg(u_t)\not=\deg(b)-2}}s\, m_{s,t}\; u_t\; b^{s-1}\right.\\
&&\\
&&+\Div(\vec Q)\Bigg)\vn b,
\end{eqnarray*}
so that the element $\sum m_{s,t}\; b^s\;  u_t
+\left(\vn b\times \vec Q\right)$ is a Casimir (element of $Z^0_o(\psi_b)$) and, according to proposition \ref{prp:H0o}, there exist some constants $\alpha_v\in\k$, for $v\in \N$, such that
\begin{equation}\label{eq:suite1}
\sum_{\stackrel{s\in \N^*, 1\leq t\leq \mu-1}{\deg(u_t)\not=\deg(b)-2}} m_{s,t}\; b^s\;  u_t
+\left(\vn b\times \vec Q\right) = \sum_{v\in \N^*}\alpha_v \; b^v
\end{equation}
(here $\alpha_0=0$, because $\deg(b)\geq 2$ and for example using (\ref{eq:finitedim}))
and
\begin{eqnarray}\label{eq:suite2}
\lefteqn{-\sum_{\stackrel{k\in \N, 0\leq \ell\leq \mu-1}{\deg(u_\ell)=\deg(b)-2}}e_{k,\ell} \;b^k\; u_\ell
+\sum_{\stackrel{s\in \N^*, 1\leq t\leq \mu-1}{\deg(u_t)\not=\deg(b)-2}}s\, m_{s,t}\; u_t\; b^{s-1}
+\Div(\vec Q)} \nonumber\\
&& \qquad \qquad \qquad \qquad \qquad \qquad \qquad \qquad\qquad = \sum_{v\in \N^*}v\,\alpha_v \; b^{v-1}.
\end{eqnarray}
Using Euler's formula in equation (\ref{eq:suite1}) to write $b=\frac{1}{\deg(b)}\vec E\times\vn b$ and secondly the exactness of the Koszul complex, we obtain the existence of an element $d\in \A'$ satisfying:
$$
\sum_{\stackrel{s\in \N^*, 1\leq t\leq \mu-1}{\deg(u_t)\not=\deg(b)-2}} \frac{m_{s,t}}{\deg(b)}\; b^{s-1}\;  u_t\; \vec E - \vec Q = \sum_{v\in \N^*}\frac{\alpha_v}{\deg(b)} \; b^{v-1}\; \vec E + d\vn b.
$$
Computing the divergence in this last equation leads to
\begin{eqnarray}\label{eq:suitediv}
\lefteqn{\sum_{\stackrel{s\in \N^*, 1\leq t\leq \mu-1}{\deg(u_t)\not=\deg(b)-2}}
 \frac{m_{s,t}}{\deg(b)}\left((s-1)\deg(b)+\deg(u_t)+2\right)\; b^{s-1}\;  u_t + \Div(\vec Q)} \nonumber\\
  &&\qquad\qquad\quad= \sum_{v\in \N^*}\frac{\alpha_v}{\deg(b)}\left((v-1)\deg(b)+2\right) \; b^{v-1} - \vn d\times\vn b.
\end{eqnarray}
Now, (\ref{eq:suite2}), together with (\ref{eq:suitediv}), give
\begin{eqnarray*}
\lefteqn{\Div(\vec Q) = }\\
&-&\sum_{\stackrel{s\in \N^*, 1\leq t\leq \mu-1}{\deg(u_t)\not=\deg(b)-2}}
 \frac{m_{s,t}}{\deg(b)}\left((s-1)\deg(b)+\deg(u_t)+2\right)\; b^{s-1}\;  u_t\\
 &+&\sum_{v\in \N^*}\frac{\alpha_v}{\deg(b)}\left((v-1)\deg(b)+2\right) \; b^{v-1} - \vn d\times\vn b=\\
&&\!\!\!\!\!\!\!\!\sum_{\stackrel{k\in \N, 0\leq \ell\leq \mu-1}{\deg(u_\ell)=\deg(b)-2}}e_{k,\ell} \;b^k\; u_\ell
-\sum_{\stackrel{s\in \N^*, 1\leq t\leq \mu-1}{\deg(u_t)\not=\deg(b)-2}}s\, m_{s,t}\; u_t\; b^{s-1}
+\sum_{v\in \N^*}v\,\alpha_v \; b^{v-1},
\end{eqnarray*}
which gives us:
\begin{eqnarray*}
\lefteqn{\sum_{\stackrel{s\in \N^*, 1\leq t\leq \mu-1}{\deg(u_t)\not=\deg(b)-2}}
 \frac{m_{s,t}}{\deg(b)}\left(-\deg(b)+\deg(u_t)+2\right)\; b^{s-1}\;  u_t}\\
 &+&\sum_{\stackrel{k\in \N, 0\leq \ell\leq \mu-1}{\deg(u_\ell)=\deg(b)-2}}e_{k,\ell} \;b^k\; u_\ell -\sum_{v\in \N^*}\frac{\alpha_v}{\deg(b)}\left(-\deg(b)+2\right) \; b^{v-1}\\
 &=&  - \vn d\times\vn b.
\end{eqnarray*}
Now, by lemma \ref{lma:sqfree2}, this leads to $e_{k,\ell}=0$, $m_{s,t}=0$ and  $\alpha_v=0$, for all $k, \ell, s, t$ and $v$.
This permits us to conclude that the sum is a direct sum and permits us to obtain the desired result.
\qed

Let us now determine the even Poisson cohomology associated to the Poisson structure
$$
\psi_b =b_y\,\theta\px\pt - b_x\, \theta\py\pt.
$$
 First of all, we give the $0$-th even Poisson cohomology space.
\begin{prp}\label{prp:H0even}
Let $b\in \A'$ be an non-constant polynomial of $\A'=\k[x,y]$.
The $0$-th even Poisson cohomology space associated to $\psi_b$ is zero:
$$
H^0_e(\psi_b)\simeq\lbrace 0\rbrace.
$$
\end{prp}
{\bf Proof.}
Under the identifications of the cochain spaces we can write:
$$
H^0_e(\psi_b)\simeq\lbrace a\in \A'\mid a\vn b=\vec0\rbrace.
$$
As $b$ is supposed to be non-constant, this gives $H^0_e(\psi_b)\simeq\lbrace 0\rbrace$.
\qed\\
\begin{prp}\label{prp:H1even}
Let $b\in \A'$ be a homogeneous and non-constant polynomial of $\A'=\k[x,y]$.
If $b$ is square-free then a basis of the first even Poisson cohomology space associated to $\psi_b$ is given by:
\begin{eqnarray*}
H^1_e(\psi_b) &\simeq&\left\lbrace\begin{array}{lcl} (0,\k[b]^2) &\hbox{ if }& \deg(b)= 1,\\
\bigoplus\limits_{i=0}^{\mu-1}\k[b]\; \left(u_i, \vec 0\right)
\;\oplus\; \k[b]\, \left(0, \vn b\right)&&\\
\qquad\oplus\k[b]\, \left(0, \vec E\right) &\hbox{ if }& \deg(b)= 2,\\
\bigoplus\limits_{i=0}^{\mu-1}\k[b]\; \left(u_i, \vec 0\right)
\;\oplus\; \k[b]\, \left(0, \vn b\right) &\hbox{ if }& \deg(b)> 2.
\end{array}
\right.
%
%
 %
 \\
 &&\\
  &\simeq& \left\lbrace\begin{array}{lcl} \k[b]\px\oplus\k[b]\py &\hbox{ if }& \deg(b)= 1,\\
\bigoplus\limits_{i=0}^{\mu-1}\k[b]\; u_i\; \theta\pt \;\oplus\; \k[b]\, \left(-b_y\px+b_x\py\right)&&\\
\qquad\oplus\k[b]\, \left(x\px+y\py\right) &\hbox{ if }& \deg(b)= 2,\\
\bigoplus\limits_{i=0}^{\mu-1}\k[b]\; u_i\; \theta\pt \;\oplus\; \k[b]\, \left(-b_y\px+b_x\py\right) &\hbox{ if }& \deg(b)> 2.
\end{array}
\right.
%
 %
 %
\end{eqnarray*}
\end{prp}
{\bf Proof.}
Let us recall that one can write
$$
H^1_e(\psi_b)\simeq \frac{\left\lbrace (r,\vec Q)\in \A'\times\A'^2\mid \vn \left(\vec Q\times\vn b\right)+\Div(\vec Q)\vn b=\vec 0\right\rbrace}{\left\lbrace\left(\vn b\times \vn c, \vec 0\right)\mid c\in \A'\right\rbrace}.
$$
Let $(r, \vec Q)\in Z^1_e(\psi_b)$ be an even $1$-cocycle. The element $\vec Q$ then satisfies the equation:
\begin{equation}\label{eq:even1cocycle}
 \vn \left(\vec Q\times\vn b\right)+\Div(\vec Q)\vn b=\vec 0.
 \end{equation}
First of all, suppose that $\deg(b)=1$. In this case, writing $\vec Q=\left(\begin{smallmatrix}Q_1\\ Q_2\end{smallmatrix}\right)$, one computes that
$$
\vec 0 = \vn \left(\vec Q\times\vn b\right)+\Div(\vec Q)\vn b=-\left(\begin{array}{cc}\vn b\times \vn Q_1\\\vn b\times \vn Q_2\end{array}\right).
$$
According to proposition \ref{prp:H0o}, this is equivalent to $Q_1\in \k[b]$ and $Q_2\in \k[b]$. Finally, because $\deg(b)=1$, one has $\mu=0$, so that lemma \ref{lma:sqfree2} proves that $H^1_e(\psi_b) \simeq (0,\k[b]^2)$ if $\deg(b)= 1$.

Now, suppose $\deg(b)\geq 2$.
 Because the operator $\vec Q\mapsto  \vn \left(\vec Q\times\vn b\right)+\Div(\vec Q)\vn b$ is homogeneous, we can suppose that $\vec Q$ is homogeneous, which means that $\vec Q$ is given by two homogeneous polynomials of the same degree.
  Equation (\ref{eq:even1cocycle})  implies that the element $\vec Q\times\vn b\in \A'$ is an odd $0$-cocycle, so that, according to proposition \ref{prp:H0o}, there exists a constant $\alpha\in \k$ and $v\in \N$ such that
\begin{equation}\label{eq:even1cocycle2}
\vec Q\times\vn b = \alpha\; b^v = \frac{\alpha}{\deg(b)}\; b^{v-1}\;\vec E\times\vn b,
\end{equation}
(notice that $v\not=0$, because $\deg(b)\geq 2$ and because of (\ref{eq:finitedim}) for example).
Because the Koszul complex is exact, this implies that there exists an element $q\in \A'$ satisfying
$$
\vec Q= \frac{\alpha}{\deg(b)}\; b^{v-1}\;\vec E + q\vn b.
$$
Now, we compute
$$
\Div(\vec Q) =  \frac{-\alpha((v-1)\deg(b)+2)}{\deg(b)}\; b^{v-1} + \vn q\times\vn b,
$$
which, together with (\ref{eq:even1cocycle}) and (\ref{eq:even1cocycle2}), lead to:
\begin{equation}\label{eq:evenH1comp}
\frac{-\alpha(-\deg(b)+2)}{\deg(b)}\; b^{v-1} + \vn q\times\vn b=0.
\end{equation}
Now, two cases have to be studied: whether $\deg(b)=2$ or $\deg(b)\not=2$.

We first assume that $b$ is a polynomial of degree $2$. Then, equation (\ref{eq:evenH1comp}) becomes simply $\vn q\times\vn b=0$, which, in view of proposition \ref{prp:H0o} implies that $q\in \k[b]$ and in this case $\vec Q= \frac{\alpha}{\deg(b)}\; b^{v-1}\;\vec E + q\vn b\in \k[b]\vec E+\k[b]\vn b$, so that we have shown that, if $\deg(b)=2$, then
$$
H^1_e(\psi_b)\subseteq \frac{\A'}{\left\lbrace\vn b\times \vn c\mid c\in \A'\right\rbrace}\;(1,\vec 0)\oplus  \left(\k[b](0,\vec E)+\k[b](0,\vn b)\right).
$$
Conversely, it is easy to see that $\k[b](0,\vec E)+\k[b](0,\vn b)\subseteq Z^1_e(\psi_b)$ in the case $\deg(b)=2$ and that the sum is a direct one, so that, according to lemma \ref{lma:sqfree2}, we can conclude that
$$
H^1_e(\psi_b)\simeq \bigoplus_{i=0}^{\mu-1}\k[b]\; u_i\;(1,\vec 0)\oplus  \k[b](0,\vec E) \oplus\k[b](0,\vn b), \quad \hbox{ if } \deg(b)=2.
$$

Let us now consider the last case where $\deg(b)\not=2$. In this case, the equation (\ref{eq:evenH1comp})
$\frac{\alpha(2-\deg(b))}{\deg(b)}\; b^{v-1} = \vn q\times\vn b$, together with lemma \ref{lma:sqfree2}, imply in particular (because $1=u_0$) that $\alpha=0$ and $\vn b\times\vn q=0$, which as above leads to $q\in \k[b]$. In this case, we then have shown that:
$$
H^1_e(\psi_b)\simeq \bigoplus_{i=0}^{\mu-1}\k[b]\; u_i\;(1,\vec 0) \oplus\k[b](0,\vn b), \quad \hbox{ if } \deg(b)\not=2.
$$
\qed\\
\begin{prp}\label{prp:H2even}
Let $b\in \A'$ be a homogeneous and non-constant polynomial of $\A'=\k[x,y]$ and $\psi_b$ be the Poisson structure as above.
If $b$ is square-free then a basis of the $2$-nd even Poisson cohomology space associated to $\psi_b$ is given by:
\begin{eqnarray*}
H^2_e(\psi_b) &\simeq&
\bigoplus_{i=0}^{\mu-1}\k\; \left(u_i, 0, \vec 0\right) \;\; \left(\simeq \A'_{sing}(b)\right)\\
&&\\
  &\simeq& \bigoplus_{i=0}^{\mu-1}\k\; u_i\; \theta\px\py.
\end{eqnarray*}
\end{prp}
{\bf Proof.}
Recall that one can write
\begin{eqnarray*}
\lefteqn{H^2_e(\psi_b)\simeq}\\
&& \!\!\!\!\!\!\!\!\!\!\! \frac{\left\lbrace (a, d, \vec C)\in \A'\times\A'\times\A'^2\mid \vn b\times \vec C=0;
\begin{array}{c}
\renewcommand{\arraystretch}{1.5cm}
d\vn b+\vn\left(\vec C\times\vn b\right)\\
\quad+\Div(\vec C)\vn b=\vec 0\end{array}
\right\rbrace}
{\left\lbrace\left(\vec Q\times\vn b, \vn b\times\vn r, r\vn b\right)\mid (r, \vec Q)\in \A'\times \A'^2\right\rbrace}.
\end{eqnarray*}
Let now $(a, d, \vec C)\in Z^2_e(\psi_b)$ be a even $2$-cocycle. Because $\vn b\times \vec C=0$ and because the Koszul complex is exact, there exists $f\in \A'$ satisfying $\vec C=f\vn b$. Then, the other cocycle condition becomes $d\vn b+\left(\vn f\times \vn b\right)\vn b=\vec 0$, so that $d=-\vn f\times \vn b$. Now, according to (\ref{eq:finitedim}), there exist $\vec Q\in \A'^2$ and some constants $\lambda_i\in \k$, for $0\leq i \leq \mu-1$, such that
$$
a = \vec Q\times \vn b+\sum_{i=0}^{\mu-1}\lambda_i u_i.
$$
We finally have:
\begin{eqnarray*}
(a, d, \vec C) &=& \left(\vec Q\times \vn b,  -\vn f\times \vn b, f\vn b\right) +  \sum_{i=0}^{\mu-1}\lambda_i\left(u_i, 0, 0\right)\\
  &\in& B^2_e(\psi_b) \oplus \bigoplus_{i=0}^{\mu-1}\k\; \left(u_i, 0, 0\right),
\end{eqnarray*}
which gives the result.
\qed\\
\begin{prp}\label{prp:H3even}
Let $b\in \A'$ be a homogeneous and non-constant polynomial of $\A'=\k[x,y]$ and $\psi_b$ be the Poisson structure defined as previously.
If $b$ is square-free then a basis of the third even Poisson cohomology space associated to $\psi_b$ is given by:
\begin{eqnarray*}
H^3_e(\psi_b) &\simeq&
\bigoplus_{i=0}^{\mu-1}\k[b]\; \left(u_i, 0, \vec 0\right)  \;\; \left(\simeq \A'_{sing}(b)\right)\\
&&\\
  &\simeq& \bigoplus_{i=0}^{\mu-1}\k[b]\; u_i\; \theta\px\py\pt.
\end{eqnarray*}
\end{prp}
\begin{proof}
We have:
\begin{eqnarray*}
\lefteqn{H^3_e(\psi_b)\simeq}\\
&& \!\!\!\!\!\!\!\!\!\!\! \frac{\left\lbrace (a, d, \vec C)\in \A'\times\A'\times\A'^2\mid \vn b\times \vec C=0;
\begin{array}{c}
\renewcommand{\arraystretch}{1.5cm}
2d\vn b+\vn\left(\vec C\times\vn b\right)\\
\quad+\Div(\vec C)\vn b=\vec 0\end{array}
\right\rbrace}
{\left\lbrace\left(\vn b\times\vn p, \vn b\times\vn s, 2s\vn b\right)\mid (p, s)\in \A'\times \A'\right\rbrace}.
\end{eqnarray*}
Assume that an element $ (a, d, \vec C)\in Z^3_e(\psi_b)$ is an even $3$-cocycle. According to lemma \ref{lma:sqfree2}, there exist $p\in \A'$ satisfying $a\in\vn b\times \vn p + \sum_{j=0}^{\mu-1}\k[b]\; u_j$.

Moreover, as we have $\vn b\times \vec C=0$ and because of the exactness of the Koszul complex, there exists $f\in \A'$ such that $\vec C= 2f\vn b$. The other cocycle condition now becomes:
$$
2d\vn b+2\left(\vn f\times \vn b\right)\vn b=\vec 0,\; \hbox{ i.e., } \; d= -\vn f\times \vn b.
$$
Finally, this leads to
\begin{eqnarray*}
(a, d, \vec C) &\in& \left(\vn b\times \vn p, -\vn f\times \vn b, 2f\vn b\right)
+ \sum_{j=0}^{\mu-1}\k[b]\; (u_j, 0, 0)\\
&\in & B^3_e(\psi_b) \oplus\bigoplus_{j=0}^{\mu-1} \k[b]\; (u_j, 0, 0),
\end{eqnarray*}
\end{proof}
Finally, we give the $n$-th Poisson cohomology associated to $\psi_b$, for $n\geq 4$.
\begin{prp}\label{prp:Hneven}
Let $b\in \A'$ be a homogeneous and non-constant polynomial of $\A'=\k[x,y]$ and let $n\in \N$ such that $n\geq 4$. Let $\psi_b$ be the Poisson structure defined by the following
$$
\psi_b =b_y\,\theta\px\pt - b_x\, \theta\py\pt.
$$
If $b$ is square-free then a basis of the $n$-th even Poisson cohomology space associated to $\psi_b$ is given by:
\begin{eqnarray*}
H^n_e(\psi_b) &\simeq&
\bigoplus_{i=0}^{\mu-1}\k\; \left(u_i, 0, \vec 0\right)  \;\; \left(\simeq \A'_{sing}(b)\right)\\
&&\\
  &\simeq& \bigoplus_{i=0}^{\mu-1}\k\; u_i\; \theta\px\py\pt^{n-2}.
\end{eqnarray*}
\end{prp}
{\bf Proof.}
As previously, we write:
\begin{eqnarray*}
H^n_e(\psi_b)\simeq
\frac{\left\lbrace (a, d, \vec C)\in \A'\times\A'\times\A'^2\mid
\begin{array}{l}
\renewcommand{\arraystretch}{1.5cm}
 \vn b\times \vec C=0;\\
(n-1)d\vn b+\vn\left(\vec C\times\vn b\right)\\
\qquad\qquad  +\Div(\vec C)\vn b=\vec 0\end{array}
\right\rbrace}
{\left\lbrace
\!\!\!\!\!\!\!\!\begin{array}{c}
\renewcommand{\arraystretch}{1.5cm}
\left(\vn b\times\vn p-(n-3)\vec R\times\vn b, \vn b\times\vn s, (n-1)s\vn b\right)\\
\qquad \qquad \qquad \qquad \qquad \qquad \qquad \qquad\mid (p, s, \vec R)\in \A'\times \A'^2
\end{array}
\right\rbrace}.
\end{eqnarray*}
%
%
Now, to determine $H^n_e(\psi_b)$, one uses the same arguments (the exactness of the Koszul complex and the equation (\ref{eq:finitedim})) and a very similar reasoning as for the computation of the space $H^2_e(\psi_b)$.
\qed\\
\begin{rem}
Notice that the determination of the odd and even Poisson cohomology spaces associated to $\psi_b$ could have been done with the hypothesis of $b$ being homogeneous replaced by $b$ being a weight-homogeneous polynomial (i.e., where the two variables $x$ and $y$ are equipped with weights which are not necessarily equal to $1$, see \cite{Pichereau1}).
\end{rem}
\section{Final remarks: codifferentials and $P_\infty$ structures}
\label{sectionPinf}

 A very important notion in mathematics is the one of deformations. Besides the notion of formal deformations mentioned in the introduction of this paper, one can consider the deformations of $\zt$-graded Poisson structures \textit{inside} the space  $\MDer(\A)$ of multiderivations of a $\zt$-graded (polynomial) algebra $\A$. This idea leads to the notion of Poisson infinity structures, which generalizes the notion of $\zt$-graded Poisson structures.
  As usual in deformation theory, this comes with the questions of equivalence of such deformations and links with cohomology.

\subsection{Codifferentials}

In this final section, as in the preliminaries, $\A$ denotes the $\zt$-graded polynomial algebra $\A=\k[x_1,\dots, x_m, \theta_1, \dots, \theta_n]$, where the variables $x_i$ are even while the variables $\theta_j$ are odd.

A \emph{codifferential} on a \zt-graded lie algebra is an element $d\in C_o$ such that $[d,d]=0$. A
\emph{$\pinf$ structure}, or \emph{Poisson infinity structure},
on  $\A$ is a codifferential in $\MDer(\A)$ with respect to the
modified Schouten bracket. In other words, $\psi=\psi^1+\cdots$ where $\psi^k\in C^k_o$
is a \pinf-structure provided that
\begin{equation*}
\sum_{k+l=n+1}\{\psi^k,\psi^l\}=0,\qquad n=1,\dots
\end{equation*}
(Notice that higher Poisson structures, but with a different grading, were already discussed in \cite{Voro}.)

\subsection{Equivalence}
Considering this generalized notion of codifferential also leads to considering the question of equivalence of such codifferentials.
Suppose that $\ph\in C^k_e$ is an even (in the bigraded sense) \zt-graded $k$-derivation. Then
$\exp(\ph):\bigwedge \A\ra \bigwedge \A$ is an automorphism of the tensor co-algebra of $\A$,
which induces an automorphism $\exp(\ph)^*$ of the coderivations of $\bigwedge\A$, which is
given by the formula
\begin{equation*}
\exp(\ph)^*(\alpha)=\exp(-\ph)\circ\alpha\circ\exp(\ph).
\end{equation*}
We have
$\exp(\ph)^*=\exp(-\ad_{\ph})$, where $\ad_{\ph}$ is defined in terms of the modified
bracket.
In other words, we have
\begin{equation*}
\exp(\ph)^*(\alpha)=\alpha+\{\alpha,\ph\}+\tfrac12\{\{\alpha,\ph\},\ph\}+\cdots,
\end{equation*}
from which it follows that $\exp(\ph)^*(\MDer)\subseteq\MDer$. Note that if $\ph\in C^1_e$,
then $\exp(\ph)$ may not be well defined, but if it is, then it is an automorphism
of $\A$. We call such an automorphism a linear automorphism. If $\ph\in C^k_e$ for some $k>1$,
then $\exp(\ph)$, which is always well defined, is called a \emph{higher order automorphism}.

Every automorphism $g$ of $\bigwedge\A$ is of the form
$g=\lambda\circ\prod_{k=2}^\infty\exp(\ph^k)$ where $\lambda$ is a linear automorphism and
$\ph^k\in C^k_e$ for $k>1$. Note that
 there is no problem with convergence of this infinite
product.
When $\lambda\in\Aut(\A)$ and $\ph^k\in\MDer(\A)$,
then we call $g$ a \emph{multi-automorphism} of $\A$. Note that
$g^*=(\prod_{k=\infty}^2\exp(-\ad_{\ph^k}))\circ\lambda^*$, and that $g^*$ is an automorphism
of $\MDer(\A)$.

If $\psi$ and $\tilde\psi$ are codifferentials,
then we say that $\psi\sim\tilde\psi$ if there is
some automorphism $g$ of $\bigwedge\A$ such that $g^*(\psi)=\tilde\psi$. In this case,
we say that $\psi$ and $\tilde\psi$ are equivalent codifferentials, or that they give
isomorphic \pinf-algebra structures on $\A$.  If  $\psi^k$ is the leading term of
$\psi$ (\ie, the first nonvanishing
term, for the exterior degree), and $\tilde\psi^l$ is the leading term of $\tilde\psi$, then if $\psi\sim\tilde\psi$,
we must
have $k=l$ and $\psi^k$ and $\tilde\psi^k$ must be \emph{linearly equivalent} codifferentials
(that is, equivalent by means of a linear automorphism).
Note that $\psi$ and $\tilde\psi$ need not be linearly equivalent.

\subsection{Deformations and cohomology}

Finally, one of the aims of studying  codifferentials is to study  the deformations (or extensions) of codifferentials, and as for the formal deformation question, there are cohomology spaces that give information about this question. More precisely, this question of deformation/extension is the following:
suppose that $\psi=\psi^k\mplus\psi^m$ is a codifferential\footnote{If $\psi^k\not=0$ is the term in $\psi$ of smallest degree $k$, then we sometimes say that $k$ is the \textit{order} of the codifferential $\psi$.}
and $\alpha=\psi^{m+r}+\phi$, where $\psi^i\in C^i$, $r\ge 1$, $\alpha\in C_o$ and $\phi$ is given by elements of $\MDer(\A)$ which are of (exterior) degree greater or equal to $m+r+1$, then under which conditions, is $d=\psi+\alpha$ itself a codifferential ?
The answer is: $d$ is a codifferential (whose first term is the initial codifferential) iff the Maurer-Cartan equation
\begin{equation*}
D(\alpha)+\tfrac12\{\alpha,\alpha\}=0
\end{equation*}
is satisfied, where $D(\ph)=\{\psi,\ph\}$ is the coboundary operator associated to $\psi$.
Since $D^2=0$, we can define the cohomology $H(\psi)$ determined by $\psi$ by
$$H(\psi)=\ker(D)/\im(D).$$ This cohomology inherits a natural grading, so we have a decomposition $H(\psi)=H_o(\psi)\oplus H_e(\psi)$.

Now suppose that the leading term of $\alpha$ has degree $m+r$, and is the leading
term of a coboundary; \ie, $\alpha+D(\beta)$ has order at least $m+r+1$. Let us
also assume that we can choose $\beta$ to have order at least 2. Then
$\exp(-\ad_\beta)(d)=\psi+\ho$ where the higher order terms have degree at least $m+r+1$.

Let $D_k$ be defined by $D_k(\phi)=\{\psi_k,\phi\}$. Since one has
$\{\psi_k,\psi_k\}=0$, $D_k^2=0$. Since the lowest order term in $D(\alpha)+\tfrac12\{\alpha,\alpha\}$
is $D_k(\psi_{m+r})$, it follows that $\psi_{m+r}$ is a $D_k$-cocycle. If $\psi_{m+r}$ is a $D_k$-coboundary,
then the leading term of $\alpha$ is automatically the leading term of a $D$-coboundary.
In particular, if
$H^n(D_k)=0$ for
all $n>m$, then this condition is automatically satisfied.
Note that  $$H^n(D_k)=\ker(D_k:C^n\ra C^{n+k-1})/\im(D_k:C^{n-k+1}\ra C^n)$$ is well defined because $D_k$ is given by a codifferential consisting of a single term. Of course, when $\psi^k$ is a $\zt$-graded Poisson structure (that is if $\psi^k$ is a single codifferential with $k=2$), then the associated cohomology is the cohomology introduced previously in this paper.

For a general codifferential $\psi$, the definition of $H^n$ is a bit complex,
because $D$ does not respect degrees of codifferentials.  Rather, we have to
consider $C(\A)=\prod_{k=0}^\infty C^k(\A)$ as a filtered complex, with
$FC^n=\prod_{k=n}^\infty C^k$, and then it is true that $D:FC^n\ra FC^{n+1}$.
Usually, we are interested in computing $H(\psi^k)$, where $\psi^k$ is the
first nonvanishing term in $\psi$.

Note that $\psi^k$ is itself a codifferential, and that the cohomology
$H(\psi^k)$ governs extensions of $\psi^k$ to a codifferential with higher
order terms. The cohomology $H(\psi^k)$ has a decomposition in the form
$H=\prod_{n=0}^\infty(H^n)$ where $H^n=Z^n/B^n$ with
\begin{align*}
Z^n&=\ker(D:C^n\ra C^{n+k-1})\\
B^n&=\im(D:C^{n-k+1}\ra C^n).
\end{align*}

 To illustrate these notions, let us talk about  some concrete examples in the $0|1$-dimensional and $1|1$-dimensional cases.

\subsection{The $0|1$-dimensional case.}
 First, on the $0|1$-dimensional polynomial algebra $\A=\k[\theta]=\k\oplus\k\theta$, analogously to the Poisson case (in section \ref{section01}), every odd element $\psi=\psi^1a_1+\psi^2a_2+\cdots$ in $\MDer(\A)$, satisfies
$[\psi,\psi]=0$, and thus $\psi$ determines a $\pinf$-algebra structure on $\A$.
(It is also possible to include a term $\psi^0a_0$ from $C^0_o$ in the definition of a $\pinf$-algebra,
but it is less conventional to do so.)

%
%
Suppose that $\psi=\psi^ka_k+\psi^{l}a_{l}+\ho$, where $k<l$,
is a $\pinf$-algebra structure of
order $k$; in other words, $a_k\ne0$. Then
$$\exp(c\ph^{l-k+1})^*(\psi)=\psi^ka_k+\psi^l(a_l+kc(-1)^{(k-1)(l-k)}a_k)+\ho,$$
so if we choose $c=-(-1)^{(k-1)(l-k)}\frac{a_l}{ka_k}$, we can eliminate the $\psi^l$ term. It follows
that $\psi\sim\psi^ka_k$. Then, applying a linear equivalence to $\psi^ka_k$, we see
that $\psi^ka_k\sim\psi^k$. It follows that up to equivalence, the structures
$\psi^k$, for $k=1\dots$ give rise to all $\pinf$-algebra structures on $\A$.
(We do not consider $\pinf$ structures with a nonzero term in $C^0$ in this paper.)

Another way to see that any codifferential is equivalent to one of the
form $\psi^k$ is to consider the cohomology associated to $\psi^k$. Indeed, if $\psi=\psi^k$ be a codifferential on $\A$, then it is easy to see that
\begin{align*}
H^n(\psi)=
\begin{cases}
\k\psi^n,&0\leq n<k-1\\
0,&n\ge k-1.
\end{cases}
\end{align*}
Notice that for deformation theory, as in the Poisson case, we normally do not include the odd $0$-cochains,
this implies that it is natural to interpret $H_o^{k-1}=Z_o^{k-1}=\k\psi^{k-1}$ and in this particular case, $H^{k-1}=H_o^{k-1}=\k\psi^{k-1}$.

Now, it is easy to see that if $H_o^n(\psi)=0$ for all $n>k$, then every
extension of a codifferential $\psi$ of degree $k$ to a codifferential of leading term
$\psi$ by adding higher order terms must be equivalent to $\psi$.

Moreover, the above calculation shows that the codifferential $\psi^k$ has deformations to all $\psi^n$ for $1\le n<k$, and
this gives the complete deformation picture on the space of all $\pinf$-algebra structures on $\A$.

\subsection{The $1|1$-dimensional case.}

On the $1|1$-dimensional polynomial algebra $\A=\k[x,\theta]$ and analogously to the Poisson case (section \ref{section11}) an odd element $\psi=\sum_{k=1}^\infty\psi^k =\sum_{k=1}^\infty (f_k(x)\theta\px\pt^{k-1}+g_k(x)\pt^k)\in \MDer(\A)$ is a codifferential precisely when either all the $f_k$ (codifferentials \textit{of the first kind}) or all the $g_k$ vanish (codifferentials \textit{of the second kind}).

In order to consider the questions above of equivalence and deformations of codifferentials in the $1|1$-dimensional case, suppose $\psi=\psi^k=g(x)\pt^k$ is a single term codifferential of the first kind ($g(x)\in \k[x]$).
Very analogous methods as in the Poisson case (section \ref{section11first}) lead to
\begin{equation*}
H^n_o(\psi)=
\begin{cases}
\k[x]\, \pt^n& n<k-1,\\
\vspace*{-0.3cm}
{}\\
\vspace*{-0.3cm}
\displaystyle{\frac{\k[x]}{(g(x))}}\, \pt^{k-1} \;&n=k-1,\\
{}\\
\displaystyle{\frac{\k[x]}{(h(x))}}\, \pt^{n} &n\ge k,
\end{cases}
\end{equation*}
and
\begin{align*}
H^n_e(\psi)=
\begin{cases}
0&n=0,\\
\k[x](kp(x)\px\pt^{n-1}+q(x)\theta\pt^n)&0<n<k,\\
\vspace*{-0.5cm}
{}\\
\displaystyle{\frac{\k[x]}{(h(x))}(kp(x)\,\px\pt^{n-1}+q(x)\theta\,\pt^n)}\;&n\ge k,
\end{cases}
\end{align*}
where $h(x)=\gcd(g(x),g'(x))$ measures the \emph{singularity} of the codifferential $\psi$.
If for deformation theory, we do not include the odd 0-cochains, then we interpret $H^{k-1}_o=Z^{k-1}_o=\k[x]\pt^{k-1}$.

As the action of a linear automorphism
on $\psi^k$ is given by
\begin{align*}
\exp(c\theta\pt)^*(g(x)\pt^k)&=\exp(ck)g(x)\pt^k\\
\exp((ax+b)\px)^*(g(x)\pt^k)&=g(rx+s)\pt^k
\end{align*}
where $r=e^{a}$ and $s=\frac{e^{a}-1}{a}b$, we have $\tilde\psi^k=\tilde g(x)\pt^k$
is equivalent to $\psi^k$ iff $\tilde g(x)=Cg(Ax+B)$ for some constants $A$, $B$ and $C$,
where $A$ and $C$ don't vanish.
It follows that the singularity of $\tilde\psi$ is
given by $\tilde h(x)=Ch(Ax+B)$.
Note that in the absence of a $c\theta\pt$ term in $\ph^1$, we only obtain automorphisms
of the form $g(x)\mapsto g(Ax+B)$. Thus the \zt-grading introduces a new kind of automorphism
of codifferentials, which we do not see in the nongraded case. Note that the
ideal $(\tilde h(x))$ generated by the transformed singularity is just
$(h(Ax+B))=\exp((ax+b)\px)(h(x))$, so
we don't see anything new on the ideal level. This says that the nature of the singularity
of the codifferential
remains unchanged under automorphisms, in other words, equivalent codifferentials of the first kind have
equivalent singularities. There is a natural isomorphism between
the quotients $\k[x]/(h(x))$ and $\k[x]/(\tilde h(x))$, in other words,
equivalent codifferentials of the first kind have the same cohomology.

Finally, using these results, we study particular examples of deformations of codifferentials.

\begin{exa}
Let $\psi=x^2\pt$, so that $h_1(x)=x$.
It follows that $H^n_o(\psi)=\k[x]/(x)\pt^n$ for $n\ge 1$. As a consequence
we can extend $\psi$ nontrivially by adding $\psi^k=c\pt^k$ for any $k>1$ and any nonzero
constant $c$.

For example, let $\psi'=x^2\pt+\pt^2$. Then if $\ph^k=a_k(x)\px\pt^{k-1}+b_k(x)\theta\pt^k$,
we have
\begin{equation*}
\{\psi',\sum\ph^k\}=x^2b_0+\sum_{k\ge0}(x^2b_{k+1}(x)-2xa_{k+1}(x)+\s{k+1}2b_k(x))\pt^{k+1}.
\end{equation*}
From this relation, we see that
$B^0_o(\psi')=x^2\k[x]$,
$B^1_o(\psi')=x\k[x]\pt$, while
$B^k_o(\psi')=\k[x]\pt^k$ for $k>1$. This last
condition follows from the fact that if $k>1$, then setting $a_k(x)=\tfrac12xb_k(x)$,
we obtain that $\{\psi,\ph^k\}=\s{k+1}2b_k\pt^{k+1}$, so every codifferential of the
first kind of degree greater than 1 is a coboundary. This means that $H^k_o(\psi')=0$ for $k>1$.
As a consequence, every extension of $\psi'$ is equivalent to $\psi'$.

On the other hand, suppose that $\psi'=x^2\pt+x\pt^2$.
Note that we have added a 2-coboundary term to $\psi$, so it may seem that we ought to obtain
something equivalent to $\psi$. In fact, this statement is true up to higher order terms,
because if $\ph^2=\tfrac12\px\pt$ we apply $\exp(\ph^2)^*$ to $\psi'$ we add $-x\pt^2$
plus higher order terms, so we see that $\psi'$ is equivalent to $\psi$ up to terms
of order 3.

Note that we cannot apply the same reasoning to
$\psi'=(x^2+x)\pt$, because an exponential of a first order term $\ph^1$ contributes
an infinite number of terms of the same degree. In fact, we already computed the effect of
such an exponential, and it can only change $x^2$ into a polynomial of the form $a(x+b)^2$,
where $a$ is a nonzero constant, and therefore we cannot obtain $x^2+x$ as the coefficient
of $\pt$ term in the exponential of $\psi$.
\end{exa}
\begin{exa}
Let $\psi=x^3\pt$. Then $H^n_o(\psi)=\displaystyle{\frac{\k[x]}{(x^2)}}\,\pt^n$ for $n\geq1$.
Let $\psi'=x^3\pt+x\pt^2$, so that $\psi'$ is a nontrivial extension of $\psi$.
One computes that $H^2(\psi')=\displaystyle{\frac{\k[x]}{(x)}}\,\pt^n$, while $H^3(\psi')=(0)$. This implies that $\psi'$ is a maximal extension of $\psi$, at least up to order $4$.
\end{exa}

In general, one can show that if we have a codifferential of the form
$\psi=\psi^{k_1}\mplus\psi^{k_m}$, where $\psi^{k_{m+1}}$ is a nontrivial
codifferential in $H^{m+1}(\psi^{k_1}\mplus \psi^{k_m})$, then there is an
upper bound on $m$. Thus any codifferential is equivalent to one with a finite
number of terms.

The versal deformation
of the $\pinf$ algebra determined by $\psi$ coincides with the
infinitesimal deformation, because the brackets of
odd coderivations of the first kind with each other always vanishes,
so there is no obstruction to the deformation.

\nocite{*}
\bibliographystyle{amsplain}
\bibliography{global}
\end{document}